\pdfoutput=1
\documentclass[10pt,twoside]{article}

\usepackage{8OSME-style}

\usepackage{tabularx}
\usepackage{overpic}
\usepackage{tikz}
\usepackage{textcomp, gensymb}

\newtheorem{lemma}{Lemma}
\newtheorem{proposition}{Proposition}
\theoremstyle{definition}
\newtheorem{definition}{Definition}
\newtheorem{algorithm}{Algorithm}

\numberwithin{equation}{section}

\theoremstyle{remark}
\newtheorem{remark}{Remark}
\theoremstyle{plain}

\usepackage{overharpoon}

\DeclareMathAlphabet{\mathcal}{OMS}{cmsy}{m}{n}

\newcommand{\MA}{\mathfrak{M}}
\newcommand{\FR}{\mathfrak{F}}
\newcommand{\dist}[1]{\lvert\overline{#1}\rvert}

\setlist[enumerate,1]{label=\textup{(\arabic*)}, ref=(\arabic*)}

\title{Triangles and quadrilaterals suffice}
\authorlist[Tsai]{Mu-Tsun Tsai}
\affiliations{Mu-Tsun Tsai \at New Taipei City, TAIWAN, \email{don.m.t.tsai@gmail.com}}

\makeatletter
  \let\runtitle\@title
  \let\runauthor\shortauthor
\makeatother

\begin{document}

\maketitle

\begin{abstract}
	We prove a conjecture proposed by Robert J. Lang over 20 years ago,
	stating that in a circle-packed origami design, any axial polygon with 5 or more sides can be reduced by properly adding a new circle.
	The proof is based on the structure of the medial axis of the region enclosed by the circles,
	and yields a concrete algorithm for finding the added circle.
	The result extends to configurations involving rivers.
\end{abstract}


\section{Introduction}
\label{sec:introduction}

In computational origami design, the \emph{circle/river method} provides a systematic approach to constructing uniaxial bases with prescribed flap lengths by representing each flap as a circle whose radius is proportional to the desired flap length, then packing these circles such that their centers lie in the given square and the circles do not overlap. A key step in this method involves filling each \emph{axial polygon}---formed by connecting the centers of mutually tangent circles---with the crease pattern of a corresponding molecule. For triangular and quadrilateral axial polygons, simple geometric constructions are known for their associated molecules. However, for polygons of order five or higher, no comparable geometric construction exists; instead, one must resort to numerical algorithms to compute the patterns of the molecules. This computational dependency poses a barrier to purely geometric origami design.

Lang observed \cite[pp. 364--366]{Lang2003}\footnote{See also \cite[pp. 379--381]{Lang2011}.} that this difficulty could potentially be circumvented through a reduction process. By judiciously adding a new circle tangent to at least three existing circles in an $N$-gon with $N > 4$, one might decompose the configuration into smaller polygons, ultimately reducing all axial polygons to triangles and quadrilaterals. If successful, this approach would enable the design of arbitrary circle-packed bases using only geometric constructions, without reliance on computational tools. In all cases examined by Lang, such a reduction was indeed possible---one could always find a circle whose addition ensures that the largest reduced polygon has at most $N-1$ sides. However, a rigorous proof that this reduction is always achievable remained an open problem.

The significance of this problem lies not in its practical utility---adding circles introduces additional flaps that may complicate the folding process---but rather in its theoretical implications for understanding the geometric structure of circle packings in origami design. The question touches on fundamental properties of tangent circle arrangements and their associated polygonal decompositions. Lang noted that while certain configurations admit circles whose addition preserves the polygon order, in all examined cases an alternative placement could be found that achieves the desired reduction. This observation suggested a deeper geometric principle at work.

In this paper, we prove that Lang's conjecture holds in full generality:

\begin{theorem}\label{thm:main}
	For any $N$--gon of touching circles with $N>4$,
	it is always possible to add a circle touching at least three others
	so that the largest reduced polygon has, at most, $N-1$ sides.
\end{theorem}

Furthermore, we show that the result extends to configurations incorporating rivers, demonstrating the robustness of the underlying geometric principle.

Our proof rests on the notion of the \emph{medial axis} of a plane domain, introduced by Blum \cite{Blum1967} in the context of shape recognition, and on the rigorous structure theory of medial axes later developed by Choi, Choi, and Moon \cite{ChoiChoiMoon1997}. The key observation is that the free region enclosed by the circles of an axial polygon has a medial axis which is a tree with exactly $N$ leaves, and that the internal vertices of this tree are precisely the centers of the circles that may be added to the configuration so as to touch at least three of its circles. This identification absorbs all the geometry of the problem at once, as Lang's conjecture reduces to a purely combinatorial property of trees, which we isolate as an elementary tree lemma (Lemma~\ref{lem:treelemma}), the only non-geometric ingredient of the proof.

The paper is organized as follows. Section~\ref{sec:preliminaries} collects the necessary definitions and background. Section~\ref{sec:proof} proves Theorem~\ref{thm:main}, together with a quantitative consequence---at most $N-4$ added circles suffice to reduce an $N$-gon completely into triangles and quadrilaterals (Corollary~\ref{cor:bound}). Section~\ref{sec:construction} turns the proof into an explicit procedure. The construction is fully algorithmic---all figures in this paper were generated by our implementation of it---and, since each added circle solves a classical Apollonius tangency problem, it is constructible with straightedge and compass, so the reduction stays within the purely geometric means whose pursuit motivated Lang's question in the first place. Section~\ref{sec:optimality} shows that the bound of $N-4$ added circles cannot be improved. Section~\ref{sec:rivers} extends the result to polygons crossed by a river, and Section~\ref{sec:multirivers} to polygons crossed by any number of rivers.


\section{Preliminaries}
\label{sec:preliminaries}

\subsection{Polygons of touching circles}

Throughout the paper, all indices of objects associated with an $N$-gon are understood cyclically with respect to $N$ (unless otherwise stated), so that $v_{N+1}=v_1$, $C_0 = C_N$, and so on.

\begin{definition}\label{def:ring}
	Let $C_1,\dots,C_N$ be a collection of circles, where $C_i$ is centered at $v_i$ and has radius $r_i>0$.
	We say that the circles form an \emph{$N$-gon of touching circles} if
	\begin{enumerate}
		\item\label{it:polygon} the centers $v_1,\dots,v_N$, in this cyclic order, are the vertices of a simple polygon $P$;
		\item\label{it:tangent} $\dist{v_iv_{i+1}} = r_i + r_{i+1}$ for all $i$, that is, consecutive circles are externally tangent; and
		\item\label{it:disjoint} $\dist{v_iv_j} > r_i + r_j$ for all non-consecutive pairs $i,j$, that is, all other pairs of circles are disjoint.
	\end{enumerate}
	We write $D_i$ for the closed disk bounded by $C_i$, and $t_i \in \overline{v_iv_{i+1}}$ for the point of tangency of $C_i$ and $C_{i+1}$.
\end{definition}

These configurations are exactly the axial polygons arising in the circle/river method, in the special case where no river is present (the case with rivers will be covered in Section~\ref{sec:rivers}). We emphasize that $P$ is not assumed to be convex. Condition~\ref{it:disjoint} excludes the degenerate case in which two non-consecutive circles also touch; this loses no generality, since such a configuration is already subdivided by the tangency into two polygons of touching circles of order at most $N-1$ each, and it suffices to apply Theorem~\ref{thm:main} to a sub-polygon of order greater than four.

Definition~\ref{def:ring} determines how the disks sit inside the polygon.

\begin{lemma}\label{lem:cover}
	In an $N$-gon of touching circles, each edge $\overline{v_iv_{i+1}}$ of $P$ is covered by $D_i \cup D_{i+1}$; moreover, $D_i$ contains no vertex of $P$ other than $v_i$ and meets no edge of $P$ other than the two incident to $v_i$.
\end{lemma}

\begin{proof}
	A point $x$ on the edge $\overline{v_iv_{i+1}}$ at distance $\alpha$ from $v_i$ satisfies either $\alpha \le r_i$, in which case $x \in D_i$, or $\dist{xv_{i+1}} = (r_i + r_{i+1}) - \alpha < r_{i+1}$, in which case $x \in D_{i+1}$. For any $j \ne i$ we have $\dist{v_iv_j} \ge r_i + r_j > r_i$, so $v_j \notin D_i$. Finally, suppose $x \in D_i$ lies on an edge $e = \overline{v_jv_{j+1}}$ not incident to $v_i$. Since $e \subseteq D_j \cup D_{j+1}$, we have $x \in D_\ell$ for some $\ell \in \{j, j+1\}$ with $\ell \ne i$, and hence $\dist{v_iv_\ell} \le \dist{v_ix} + \dist{xv_\ell} \le r_i + r_\ell$. If $C_i$ and $C_\ell$ are non-consecutive this contradicts Definition~\ref{def:ring}\ref{it:disjoint}. If they are consecutive, then equality holds throughout, so $x$ lies on the segment $\overline{v_iv_\ell}$---an edge of $P$ incident to $v_i$---at distance $r_i$ from $v_i$, an interior point of that edge since $0 < r_i < r_i + r_\ell$. The edge $e$ thus meets the edge $\overline{v_iv_\ell}$ at an interior point of the latter, contradicting the simplicity of $P$.
\end{proof}

For a point $p$ of the plane, let
\begin{equation}
	d_i(p) = \dist{pv_i} - r_i
\end{equation}
denote the signed distance from $p$ to the circle $C_i$; thus $d_i(p) > 0$ if and only if $p$ lies outside $D_i$.

\begin{definition}\label{def:add}
	Given an $N$-gon of touching circles and a circle $C$ with center $p$ in the interior of $P$ and radius $\rho>0$,
	we say that $C$ \emph{can be added} to the configuration if $d_i(p) \ge \rho$ for all $i$,
	and that it \emph{touches} $C_i$ if $d_i(p)=\rho$.
	We call $C$ \emph{admissible} if it can be added and touches at least three of the circles.
\end{definition}

Note that an added circle is automatically contained in $P$. Indeed, by Lemma~\ref{lem:cover}, every edge of $P$ is covered by the two disks centered at its endpoints, so a disk that avoids all the $D_i$ and whose center lies in $P$ cannot reach the boundary of $P$.

The reduction will only ever add admissible circles. When an admissible circle $C$ touching $C_{i_1},\dots,C_{i_k}$ (with $i_1 < \cdots < i_k$, so $k \ge 3$, and the subscript $j$ of $i_j$ understood cyclically with respect to $k$) is added, the segments connecting $p$ to the centers $v_{i_1},\dots,v_{i_k}$ subdivide $P$ into $k$ smaller polygons of touching circles, as illustrated in Figure~\ref{fig:split}(a). The polygon between the tangencies at $C_{i_j}$ and $C_{i_{j+1}}$ has $g_j + 2$ sides, where
\begin{equation}\label{eq:gap}
	g_j = i_{j+1} - i_j \pmod N, \qquad \sum_{j=1}^{k} g_j = N.
\end{equation}
We call $g_1,\dots,g_k$ the \emph{gaps} of the added circle. Every reduced polygon has at most $N-1$ sides if and only if every gap satisfies $g_j \le N-3$. In particular, a circle nestled into a ``corner'' of the configuration, touching three consecutive circles $C_{i-1}$, $C_i$, $C_{i+1}$, has gaps $(1,1,N-2)$ and produces two triangles together with another $N$-gon, as in Figure~\ref{fig:split}(b); this is precisely the non-reducing placement observed by Lang, and the content of Theorem~\ref{thm:main} is that a reducing placement always exists.

\begin{figure}[tb]
	\centering
	\subfloat[]{\includegraphics[width=0.46\textwidth]{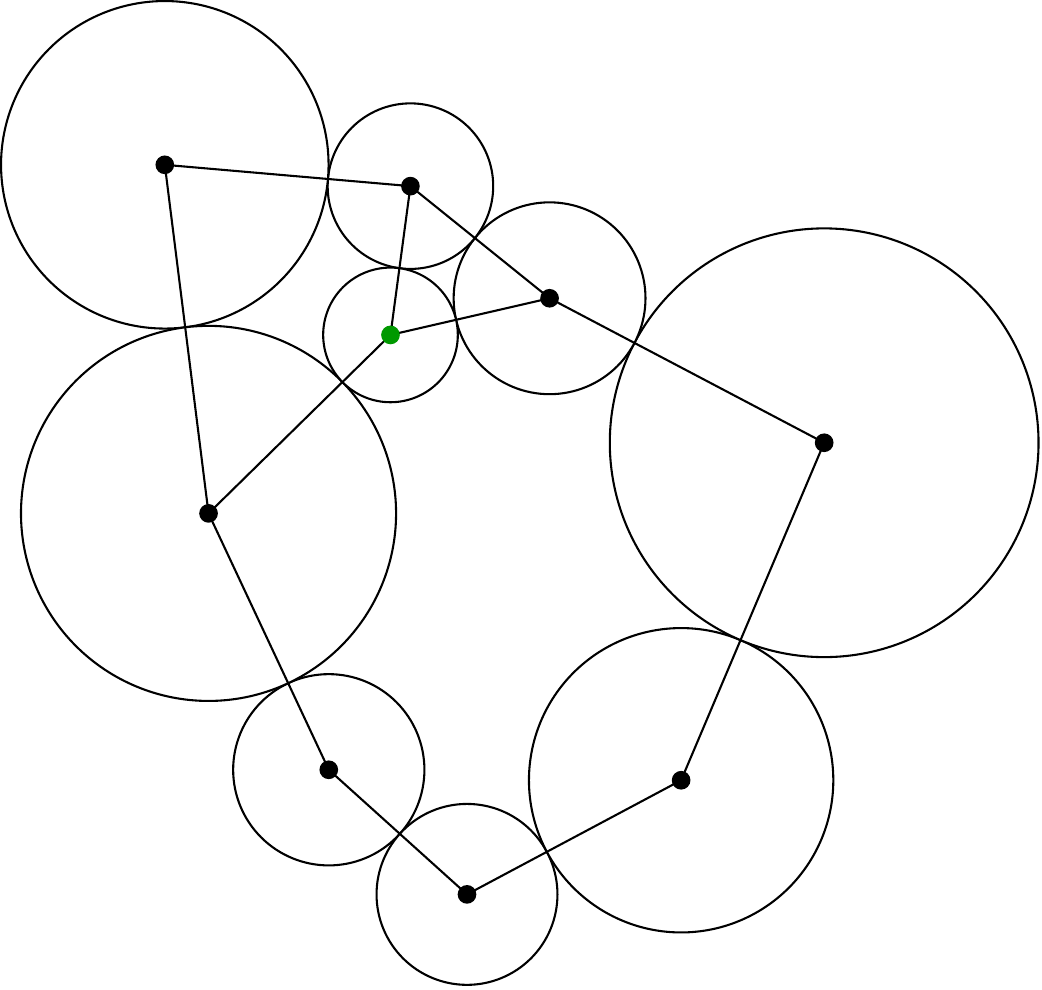}}\hspace{0.04\textwidth}
	\subfloat[]{\includegraphics[width=0.46\textwidth]{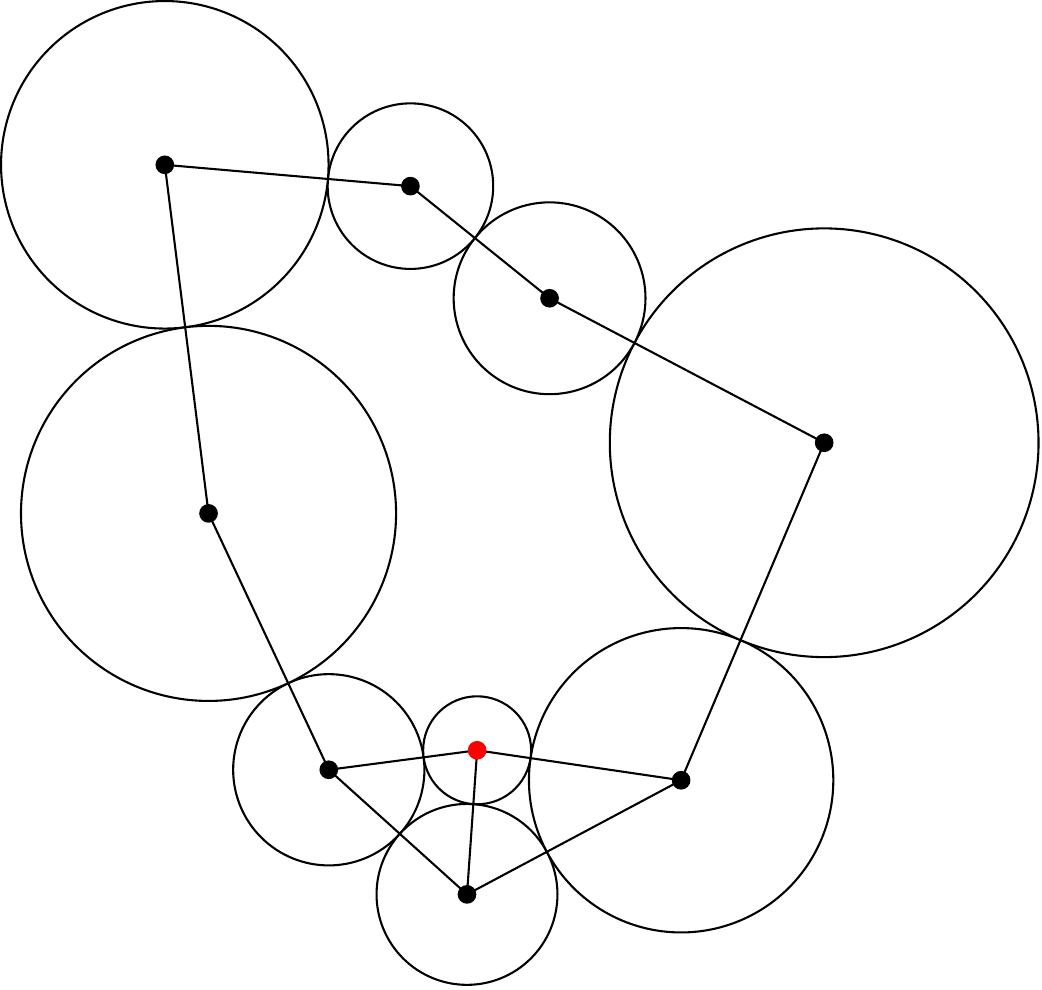}}
	\caption{Two subdivision steps applied to the same octagon of touching circles ($N = 8$). A circle (center marked) is added, and the segments from its center to the centers of the touched circles subdivide the octagon. (a) The added circle touches three of the circles with gaps $1$, $2$, $5$, producing a triangle, a quadrilateral, and a heptagon---each again a polygon of touching circles. (b) A circle nestled into a corner, touching three consecutive circles. The gaps are $(1,1,6)$, and the subdivision produces two triangles together with another octagon.}
	\label{fig:split}
\end{figure}

\subsection{The medial axis}

Let $\Omega \subset \mathbb{R}^2$ be a compact connected set which is the closure of its interior. A closed disk $\mathcal{D} \subseteq \Omega$ is \emph{maximal} if it is not contained in any strictly larger disk contained in $\Omega$. The \emph{medial axis} of $\Omega$ is
\begin{equation}
	\MA(\Omega) = \{\, p \in \Omega : p \text{ is the center of a maximal disk} \,\},
\end{equation}
where, by convention, the \emph{sharp corners} of $\partial\Omega$---corner points whose interior angle is less than $\pi$, at which the maximal disk degenerates to a point---are included in $\MA(\Omega)$. For $p \in \MA(\Omega)$ with its corresponding maximal disk $\mathcal{D}_p$, the set $\mathcal{D}_p \cap \partial\Omega$ is called the \emph{contact set} of $p$, and its connected components the \emph{contact components}.

The medial axis was introduced by Blum \cite{Blum1967} as a tool for shape description. A rigorous structure theory was developed by Choi, Choi, and Moon \cite{ChoiChoiMoon1997} for the class of connected domains bounded by finitely many disjoint simple closed curves (allowing holes), each curve consisting of finitely many real-analytic pieces. All domains appearing in this paper belong to this class, and moreover all their boundary corners are sharp. We summarize the facts we need. Here a \emph{geometric graph} is a subset of the plane which is a connected finite graph whose vertices are points and whose edges are real-analytic curves of finite length.

\begin{theorem}[Choi--Choi--Moon \cite{ChoiChoiMoon1997}]\label{thm:ccm}
	Let $\Omega$ be a domain of the above class, other than a disk. Then the following hold.
	\begin{enumerate}
		\item\label{ccm:graph} $\MA(\Omega)$ is a geometric graph {\upshape(their Theorem 8.2)}; if $\Omega$ is simply connected, it is a tree {\upshape(by their Corollary 8.1)}.
		\item\label{ccm:degree} The degree of a point $p \in \MA(\Omega)$ as a vertex of this graph equals the number of contact components of $p$ {\upshape(their Section 8.2)}.
		\item\label{ccm:leaf} Every vertex of degree one is either a sharp corner of $\partial\Omega$, or the center of a maximal disk which osculates $\partial\Omega$ along its unique contact component {\upshape(their Theorem 8.2)}.
		\item\label{ccm:beta} For every point $q \in \partial\Omega$ other than a corner, there is a unique maximal disk whose contact set contains $q$; writing $\beta(q)$ for its center, the map $\beta$ is continuous away from the corners {\upshape(their Theorem 7.2)}.
	\end{enumerate}
\end{theorem}


\section{Proof of the main theorem}
\label{sec:proof}

Fix an $N$-gon of touching circles as in Definition~\ref{def:ring}, and define its \emph{free region}
\begin{equation}
	\FR \;=\; \overline{\,P \setminus (D_1 \cup \cdots \cup D_N)\,},
\end{equation}
the closure of the part of the polygon not covered by the disks. See Figure~\ref{fig:tree}.

\begin{figure}[tb]
	\centering
	\includegraphics[width=0.5\textwidth]{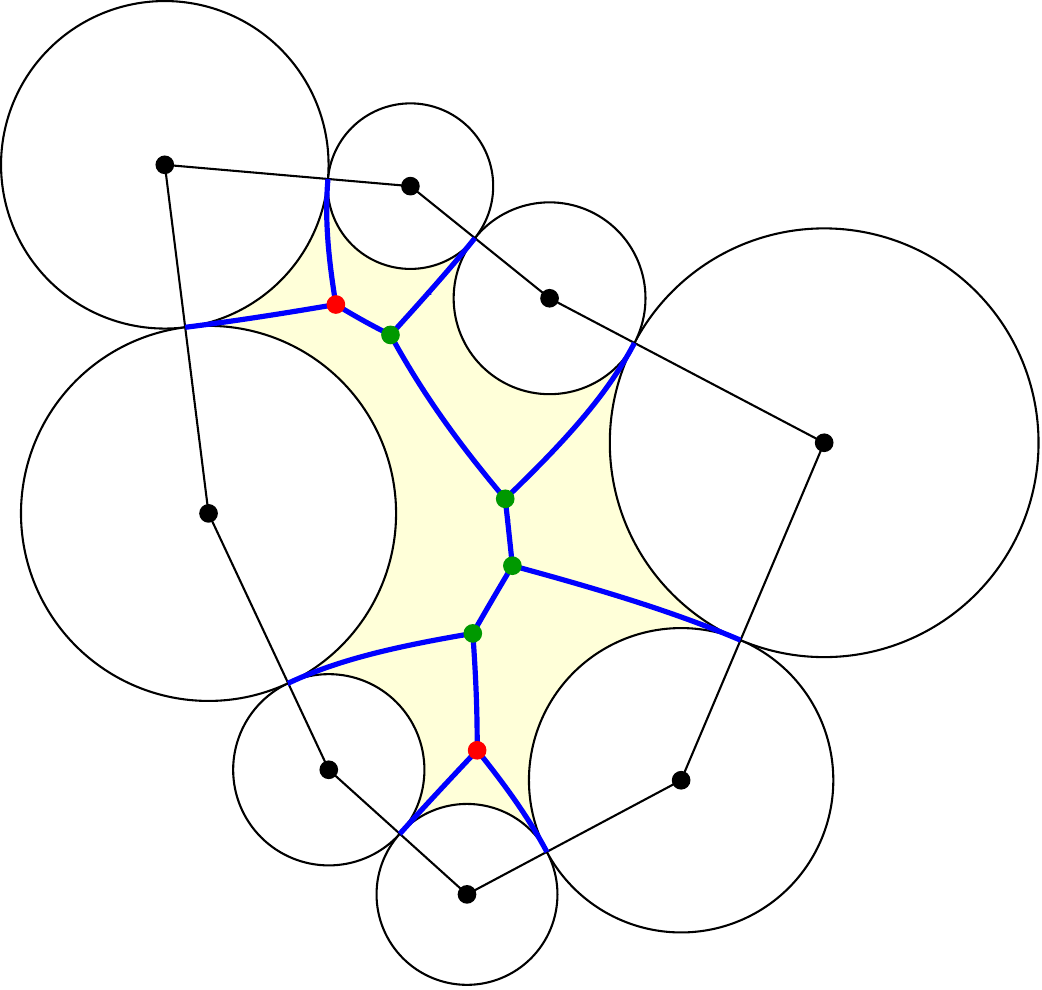}
	\caption{The free region $\FR$ (yellow) and its medial axis (blue) for the octagon of touching circles of Figure~\ref{fig:split}. The medial axis is a tree whose $N$ leaves are the tangency points $t_i$; its edges are arcs of Apollonius hyperbolas, and its internal vertices are exactly the centers of the admissible circles. Internal vertices are marked green if all their gaps are at most $N-3$, so that adding the corresponding circle leaves only polygons with at most $N-1$ sides, and red otherwise; Theorem~\ref{thm:main} asserts that a green vertex always exists. The additions of Figure~\ref{fig:split}(a) and (b) are those at a green and at a red vertex of this tree, respectively.}
	\label{fig:tree}
\end{figure}

\begin{lemma}\label{lem:domain}
	With notation as above:
	\begin{enumerate}
		\item\label{it:arc} $C_i \cap P$ is a single arc $\mathcal{A}_i$, running from $t_{i-1}$ to $t_i$ and spanning the interior angle of $P$ at $v_i$;
		\item\label{it:jordan} the arcs $\mathcal{A}_1,\dots,\mathcal{A}_N$ are pairwise disjoint except that consecutive arcs share a tangency point, and their union is a simple closed curve bounding $\FR$. Consequently $\FR$ is a simply connected domain of the class of Theorem~\ref{thm:ccm}, and the only corners of $\partial\FR$ are the $N$ points $t_1,\dots,t_N$, at which two arcs meet in a cusp of interior angle zero---a sharp corner;
		\item\label{it:sector} $D_i \cap P$ is the closed circular sector of $D_i$ spanning the interior angle of $P$ at $v_i$; in particular, for every $x \in D_i \cap P$, the segment $\overline{v_ix}$ lies in $P$.
	\end{enumerate}
\end{lemma}

\begin{proof}
	\ref{it:arc} By Lemma~\ref{lem:cover}, the circle $C_i$ meets $\partial P$ only on the two edges incident to $v_i$; on each of these it meets exactly one point, namely the point at distance $r_i$ from $v_i$, which is a tangency point ($t_{i-1}$ or $t_i$) because the edge has length $r_i + r_{i+1} > r_i$ (respectively $r_{i-1}+r_i > r_i$). A circle crossing the boundary of $P$ exactly twice intersects $P$ in a single arc, namely the one emanating from the interior angle at $v_i$.

	\ref{it:jordan} Non-consecutive arcs are disjoint because the corresponding disks are; consecutive arcs share exactly the tangency point of their disks. The concatenation $\mathcal{A}_1\mathcal{A}_2\cdots \mathcal{A}_N$ is therefore a simple closed curve $\gamma$. To see $\partial\FR \subseteq \gamma$, let $x \in \partial\FR$. Since $P \setminus (D_1 \cup \cdots \cup D_N) \subseteq \FR$, the complement of $\FR$ is covered by the exterior of $P$ together with the disks; hence $x$, lying in the closure of this complement as well as in $\FR \subseteq P$, belongs to $\partial P \cup D_1 \cup \cdots \cup D_N$, and therefore to some disk $D_i$, the boundary $\partial P$ being itself covered by the disks (Lemma~\ref{lem:cover}). On the other hand, $x$ lies in no open disk, the open disks being disjoint from $P \setminus (D_1 \cup \cdots \cup D_N)$ and hence from its closure $\FR$. Thus $x \in C_i \cap P = \mathcal{A}_i \subseteq \gamma$. Conversely, each point of an arc $\mathcal{A}_i$ is a limit both of points of $\operatorname{int} D_i$ (removed from $\FR$) and of free points, so $\gamma \subseteq \partial\FR$. Finally, the complement of $\FR$ in the plane is precisely the union of the exterior of $P$ with the open disks $\operatorname{int}D_i$. Indeed, the exterior and the open disks are disjoint from $\FR$, as noted above; conversely, a point of the complement lying in $P$ belongs to some disk $D_i$ by the covering argument above, and not to the circle $C_i$ itself, whose points in $P$ make up the arc $\mathcal{A}_i \subseteq \FR$, so it lies in $\operatorname{int}D_i$. Each open disk is connected and meets the exterior of $P$, because the interior angle of $P$ at $v_i$ is less than $2\pi$, so that $D_i$ is not contained in $P$. The complement is therefore connected, and $\FR$ is the closed Jordan domain bounded by $\gamma$. In particular $\FR$ is simply connected, and it is not a disk. Its boundary consists of $N$ real-analytic pieces. At $t_i$ the two arcs $\mathcal{A}_i$ and $\mathcal{A}_{i+1}$ are tangent to the common tangent line of $C_i$ and $C_{i+1}$ and curve away from $\FR$ on either side, so the interior angle of $\FR$ at $t_i$ is zero.

	\ref{it:sector} Inside $D_i$, each of the two edges incident to $v_i$ is a radius---it emanates from the center and leaves the disk at distance $r_i$, the edge being longer than $r_i$. The two radii cut $D_i$ into two closed circular sectors, and by Lemma~\ref{lem:cover} these radii are the only portions of $\partial P$ inside $D_i$; hence the interior of each sector, being connected and disjoint from $\partial P$, lies entirely inside or entirely outside $P$, and near $v_i$ the sector spanning the interior angle is the one inside. Thus $D_i \cap P$ is that closed sector, and it contains the radial segment $\overline{v_ix}$ for each of its points $x$.
\end{proof}

We now analyze the medial axis $\MA = \MA(\FR)$ of the free region. In the following, the tree $\MA$ is understood combinatorially. Vertices of degree two, which merely subdivide edges, are suppressed, and by a \emph{vertex} we mean a leaf or a point of degree at least three. This convention affects neither the leaves nor the components of $\MA \setminus \{p\}$ at any vertex $p$.

\begin{lemma}\label{lem:tree}
	Let $\MA = \MA(\FR)$. Then:
	\begin{enumerate}
		\item\label{it:leaves} $\MA$ is a finite tree whose leaves are precisely the tangency points $t_1,\dots,t_N$;
		\item\label{it:internal} every non-leaf vertex $p$, of degree $k$, lies in the interior of $\FR$, and its maximal disk $\mathcal{D}_p$ has positive radius $\rho$ and touches exactly $k$ of the circles, say $C_{i_1},\dots,C_{i_k}$ with $i_1 < \cdots < i_k$ (the subscript $j$ of $i_j$ understood cyclically with respect to $k$), each at a single point $q_{i_j}$ collinear with $p$ and $v_{i_j}$. In particular, the circle $\partial \mathcal{D}_p$ is admissible, touching exactly the circles $C_{i_1},\dots,C_{i_k}$;
		\item\label{it:branches} with notation as in \ref{it:internal}, the components $\MA_1,\dots,\MA_k$ of $\MA \setminus \{p\}$ can be labeled so that the leaves lying in $\MA_j$ are precisely $t_{i_j}, t_{i_j+1}, \dots, t_{i_{j+1}-1}$; their number is the gap $g_j = i_{j+1}-i_j$.
	\end{enumerate}
\end{lemma}

\begin{proof}
	\ref{it:leaves} By Lemma~\ref{lem:domain}\ref{it:jordan} and Theorem~\ref{thm:ccm}\ref{ccm:graph}, $\MA$ is a finite tree. By Theorem~\ref{thm:ccm}\ref{ccm:leaf}, a leaf of $\MA$ is either a sharp corner of $\partial\FR$---that is, one of the $t_i$---or the center of a maximal disk osculating the boundary along a single contact component. The latter does not occur. At any point $q$ in the interior of an arc $\mathcal{A}_i$, the osculating circle of the boundary is $C_i$ itself, whose disk lies outside $\FR$; a maximal disk of $\FR$ tangent to $\mathcal{A}_i$ at $q$ lies on the opposite side of the common tangent line, so it cannot osculate the arc, and a contact along a subarc of $\mathcal{A}_i$ would force the maximal disk to coincide with $D_i$. Conversely, each $t_i$ is a leaf. In a sufficiently small neighborhood of $t_i$, the boundary of $\FR$ consists only of pieces of $\mathcal{A}_i$ and $\mathcal{A}_{i+1}$, so a maximal disk centered there must touch both, placing its center on the single analytic curve $\{\,x : d_i(x)=d_{i+1}(x)\,\}$---one branch of a hyperbola with foci $v_i$ and $v_{i+1}$, passing through $t_i$. Hence exactly one edge of $\MA$ ends at $t_i$.

	\ref{it:internal} A non-leaf vertex $p$, of degree $k$, is not a corner of $\partial\FR$, since corners are leaves by part~\ref{it:leaves}; and a maximal disk centered on $\partial\FR$ has radius zero, so $p \in \operatorname{int}\FR$ and $\rho = \min_i d_i(p) > 0$. By Theorem~\ref{thm:ccm}\ref{ccm:degree}, the contact set of $p$ has exactly $k$ components. We claim no contact point is a corner $t_i$. If $\dist{pt_i} = \rho$, then, since $t_i \in C_i$,
	\[
		d_i(p) \;\le\; \dist{pt_i} + \dist{t_iv_i} - r_i \;=\; \rho,
	\]
	while $\mathcal{D}_p \subseteq \FR$ forces $d_i(p) \ge \rho$; so equality holds throughout, and $p$ lies on the ray from $v_i$ through $t_i$---that is, on the line of the edge $\overline{v_iv_{i+1}}$---at distance $r_i + \rho$ from $v_i$. But then, the edge having length $r_i + r_{i+1}$, we get $\dist{pv_{i+1}} = \lvert (r_i+\rho) - (r_i+r_{i+1}) \rvert = \lvert \rho - r_{i+1} \rvert$, and so $d_{i+1}(p) = \lvert \rho - r_{i+1} \rvert - r_{i+1} < \rho$, contradicting $\mathcal{D}_p \subseteq \FR$. Thus every contact point lies in the interior of some arc $\mathcal{A}_i$. For such a contact point $q \in \mathcal{A}_i$ we have $\rho = \dist{pq} \ge d_i(p) \ge \rho$, so $q$ is the point of $C_i$ nearest to $p$, which is unique and collinear with $p$ and $v_i$. In particular each circle carries at most one contact point, the $k$ contact components are single points $q_{i_1},\dots,q_{i_k}$ on $k$ distinct circles, and $d_i(p) \ge \rho$ for every $i$ with equality exactly at the touched circles.

	\ref{it:branches} Let $S_j \subset \partial\FR$ denote the stretch of the boundary running from $q_{i_j}$ to $q_{i_{j+1}}$ in the cyclic order, with the two endpoints excluded; walking from $q_{i_j} \in \mathcal{A}_{i_j}$ to $q_{i_{j+1}} \in \mathcal{A}_{i_{j+1}}$, one passes exactly the corners $t_{i_j}, t_{i_j+1}, \dots, t_{i_{j+1}-1}$---$g_j$ of them in total. Consider the map $\beta$ of Theorem~\ref{thm:ccm}\ref{ccm:beta}, extended to the corners by letting $\beta(t_m) = t_m$; this extension is continuous, since the inscribed radius at a boundary point tends to zero as the point approaches a cusp, being bounded by the width of the wedge between the two tangent disks. Thus $\beta(S_j)$ is a connected subset of $\MA$. Moreover, $\beta$ never takes the value $p$ on $S_j$. If $\beta(q) = p$ for a boundary point $q$, then $q$ lies in the contact set of $p$, that is, $q$ is one of the points $q_{i_1},\dots,q_{i_k}$---and these are excluded from $S_j$. The connected set $\beta(S_j)$ therefore avoids $p$ and lies in a single component of $\MA \setminus \{p\}$; call it $\MA_{(j)}$. Now every component $X$ of $\MA \setminus \{p\}$ contains at least one leaf of $\MA$ (walk within $X$ away from $p$ as far as possible), every leaf $t_m$ lies in exactly one stretch $S_j$ (contact points are not corners), and $t_m = \beta(t_m) \in \MA_{(j)}$. Hence $j \mapsto \MA_{(j)}$ is a surjection from $k$ stretches onto the $k$ components of $\MA \setminus \{p\}$, and is therefore a bijection; and the leaves in $\MA_{(j)}$ are exactly those in $S_j$.
\end{proof}

The combinatorial heart of the proof, and its only non-geometric ingredient, is the following elementary fact about trees, which we call the \emph{tree lemma}.

\begin{lemma}[tree lemma]\label{lem:treelemma}
	Every finite tree $T$ with $L \ge 5$ leaves has a vertex $p$ of degree at least three such that every component of $T \setminus \{p\}$ contains at most $L-3$ leaves of $T$.
\end{lemma}

\begin{proof}
	Since $L \ge 3$, the tree $T$ is not a path, so it has a vertex of degree at least three.

	First suppose some vertex $p$ has degree $k \ge 4$. The components of $T\setminus\{p\}$ each contain at least one leaf of $T$, and the $L$ leaves of $T$ are distributed among all $k$ of them; hence each component contains at most $L - (k-1) \le L-3$ leaves of $T$.

	Otherwise, every vertex has degree at most $3$; call those of degree exactly $3$ the \emph{branch} vertices. Let $n_1 = L$, $n_2$, and $n_3$ denote the numbers of vertices of degree $1$, $2$, and $3$, respectively. Counting degrees over the edges in two ways gives $2(n_1 + n_2 + n_3 - 1) = n_1 + 2n_2 + 3n_3$, so $n_3 = n_1 - 2$, the count $n_2$ cancelling.
	
	For a branch vertex, we say that a leaf is \emph{attached} to it, if there is no other branch vertex in between them. If every branch vertex had at least two attached leaves, then $n_1 \ge 2n_3 = 2n_1 - 4$, forcing $n_1 \le 4$ and contradicting $L \ge 5$. Hence some branch vertex $p$ has at most one attached leaf. 
	
	Each component of $T \setminus \{p\}$ either contains no branch vertex---being then a path from $p$ to the attached leaf, and there is at most one such component---or contains a branch vertex, and with it at least two leaves of $T$. Therefore, for each component, the other two components together contain at least $1+2 = 3$ leaves of $T$, so the component itself contains at most $L-3$.
\end{proof}

\begin{lemma}\label{lem:split}
	Let $p$ be a non-leaf vertex of $\MA$, of degree $k$, touching $C_{i_1},\dots,C_{i_k}$ as in Lemma~\ref{lem:tree}\ref{it:internal}, and add the circle $C = \partial \mathcal{D}_p$. Then the $k$ segments $\overline{pv_{i_j}}$ lie in $P$, meet one another only at $p$, and subdivide $P$ into $k$ simple polygons $P_1,\dots,P_k$, where $P_j$ has vertices $p, v_{i_j}, v_{i_j+1}, \dots, v_{i_{j+1}}$. Together with the added circle and the circles at its vertices, each $P_j$ is a $(g_j+2)$-gon of touching circles.
\end{lemma}

\begin{proof}
	By Lemma~\ref{lem:tree}\ref{it:internal}, the points $p$, $q_{i_j}$, $v_{i_j}$ are collinear, so $\overline{pv_{i_j}} = \overline{pq_{i_j}} \cup \overline{q_{i_j}v_{i_j}}$ is a straight segment; the first part is a radius of $\mathcal{D}_p \subseteq \FR \subseteq P$, and the second lies in $P$ by Lemma~\ref{lem:domain}\ref{it:sector}. Being straight segments emanating from the common interior point $p$ toward distinct vertices of $P$, they meet only at $p$, and since their endpoints $v_{i_1},\dots,v_{i_k}$ occur on $\partial P$ in this cyclic order, the segments subdivide $P$ into $k$ simple polygons with the stated vertex lists. In $P_j$, consecutive circles among $C_{i_j},\dots,C_{i_{j+1}}$ are tangent as before; the new edges satisfy $\dist{pv_{i_j}} = \rho + r_{i_j}$ and $\dist{pv_{i_{j+1}}} = \rho + r_{i_{j+1}}$, so $C$ is tangent to the circles at both ends; and all remaining pairs in $P_j$ are disjoint---pairs of old circles by Definition~\ref{def:ring}\ref{it:disjoint}, and $C$ from any old circle $C_m$ with $i_j < m < i_{j+1}$ because $d_m(p) > \rho$ strictly, as $C_m$ is not among the touched circles. Hence $P_j$ is a polygon of touching circles with $g_j + 2$ vertices.
\end{proof}

\begin{proof}[Proof of Theorem~\ref{thm:main}]
	Let $N > 4$. By Lemma~\ref{lem:tree}\ref{it:leaves}, the medial axis $\MA$ of the free region is a finite tree with $N \ge 5$ leaves. By the tree lemma (Lemma~\ref{lem:treelemma}), $\MA$ has a non-leaf vertex $p$ such that every component of $\MA\setminus\{p\}$ contains at most $N-3$ leaves. Lemma~\ref{lem:tree}\ref{it:internal} then provides an admissible circle $C = \partial \mathcal{D}_p$, and by Lemma~\ref{lem:tree}\ref{it:branches} its gaps satisfy $g_j \le N-3$ for all $j$. By Lemma~\ref{lem:split}, adding $C$ subdivides $P$ into polygons of touching circles with $g_j + 2 \le N-1$ sides each.
\end{proof}

Since each polygon produced by the subdivision is again a polygon of touching circles, the reduction can be iterated. The following quantitative consequence is worth recording.

\begin{corollary}\label{cor:bound}
	Every $N$-gon of touching circles can be subdivided into triangles and quadrilaterals of touching circles by adding at most $N-4$ circles.
\end{corollary}

\begin{proof}
	We argue by induction on $N$ that at most $\max(0, N-4)$ added circles suffice. For $N \le 4$ there is nothing to prove. For $N \ge 5$, Theorem~\ref{thm:main} provides a circle whose addition subdivides the polygon into $k \ge 3$ reduced polygons with $g_j + 2$ sides, where $\sum_j g_j = N$ and $g_j \le N-3$; by the induction hypothesis, each of them can be subdivided further using at most $\max(0, g_j - 2)$ added circles. The total number of added circles is therefore at most
	\[
		1 + \sum_j \max(0, g_j-2) \;=\; 1 + (N - h) - 2(k-h) \;=\; 1 + N - 2k + h,
	\]
	where $h$ is the number of gaps equal to $1$---such gaps contribute nothing, while the remaining $k - h$ gaps sum to $N - h$ and each is shortened by $2$. We claim $2k - h \ge 5$, so that the total is at most $N - 4$. If $h = k$, then every gap is $1$ and $2k - h = k = N \ge 5$. If $k = 3$, then $h \le 1$---two gaps equal to $1$ would force the third to be $N - 2 > N-3$---and $2k - h \ge 5$. And if $k \ge 4$ with $h \le k-1$, then $2k - h \ge k + 1 \ge 5$.
\end{proof}


\section{The explicit construction}
\label{sec:construction}

The proof of Theorem~\ref{thm:main} is constructive, and we record the resulting procedure explicitly as Algorithm~\ref{alg:reduction}. The edges of $\MA$ are, by the proof of Lemma~\ref{lem:tree}, arcs of the bisector curves $\{d_i = d_j\}$, each of which is a branch of a hyperbola with foci $v_i$ and $v_j$, the \emph{Apollonius hyperbola} of the two circles. The medial axis of $\FR$ is thus the portion, inside the free region, of the Apollonius diagram (the additively weighted Voronoi diagram \cite{LeeDrysdale1981}) of the $N$ circles.

\begin{algorithm}[reduction step]\label{alg:reduction}
	\emph{Input:} an $N$-gon of touching circles with $N > 4$. \emph{Output:} an added circle leaving only polygons with at most $N-1$ sides.
	\begin{enumerate}
		\item Maintain a cyclic list of bisector arcs, initialized with the $N$ arcs of consecutive pairs $(C_i, C_{i+1})$, each emanating from its tangency point $t_i$.
		\item While more than three arcs remain, search the adjacent pairs of arcs for a \emph{valid} intersection point $p$. For adjacent arcs of pairs $(C_a, C_b)$ and $(C_b, C_c)$, the two bisector curves cross in at most two points, and such a point $p$ is valid if it lies inside $P$ with $d_m(p) \ge d_a(p)$ for every circle $C_m$ not among the pairs; not every adjacent pair admits a valid point, but some pair always does. Record a valid $p$ as an internal vertex of $\MA$, touching $C_a, C_b, C_c$, and replace the two arcs by the bisector arc of $(C_a, C_c)$ emanating from $p$. When three arcs remain, record their common valid intersection as the last vertex.
		\item For each recorded vertex, read off its gaps from the touched indices by equation~\eqref{eq:gap}; return the maximal disk of any vertex whose gaps are all at most $N-3$, which exists by Theorem~\ref{thm:main}.
	\end{enumerate}
\end{algorithm}

Each merge in step (2) contracts a pair of adjacent leaves of the current tree into their common neighbor, so the loop enumerates the internal vertices of $\MA$ (a vertex of degree $k$ arising as $k-2$ coincident merges) and terminates after at most $N-3$ merges. We justify the two assertions made in step (2). An adjacent pair need not admit a valid point because the two bisector curves may well cross at a point that is not a vertex of $\MA$ at all, which happens when one of the two arcs runs into the arc of a third circle first. The validity test is what detects this. Validity is exactly the condition that the disk at the intersection avoids all other circles, so, as in Lemma~\ref{lem:tree}\ref{it:internal}, the valid points are precisely the internal vertices of $\MA$; at a premature intersection $p$, the intruding third circle is nearer to $p$ than $d_a(p)$, and $p$ fails the test. And some pair always admits a valid point. In the current tree, the second vertex of a longest path has at most one branch that is not a single leaf, so at least two of its leaf branches are cyclically adjacent, and these two arcs intersect validly at that vertex. A straightforward implementation---ours, which produced all the figures in this paper---spends $O(N)$ candidate checks of cost $O(N)$ per merge, hence $O(N^3)$ time in the worst case. Asymptotically better bounds follow from computational geometry. The Apollonius diagram of $N$ circles can be computed in $O(N \log N)$ time by Fortune's sweepline algorithm \cite{Fortune1987}, and step (3) is linear, so a single reduction step costs $O(N \log N)$, and the full subdivision of Corollary~\ref{cor:bound} costs $O(N^2 \log N)$.

Since the motivation for reducing to triangles and quadrilaterals is that those molecules admit direct geometric constructions, it is worth recording that the reduction itself demands nothing beyond straightedge and compass. An internal vertex $p$ of $\MA$ and the radius $\rho$ of its maximal disk satisfy $\dist{pv_i} = \rho + r_i$ for each touched circle $C_i$. For any three of the touched circles, this is the classical tangency problem of Apollonius. Although the equations are quadratic, squaring and subtracting the equation of one touched circle $C_a$ from those of two others cancels all quadratic terms. Writing $\overrightharpoon{x}$ for the position vector of a point $x$ in a fixed Cartesian frame, what remains is the pair of \emph{linear} equations
	\[
		2\bigl(\overrightharpoon{v_i} - \overrightharpoon{v_a}\bigr)\cdot \overrightharpoon{p} + 2(r_i - r_a)\,\rho \;=\; \bigl(\lvert\overrightharpoon{v_i}\rvert^2 - r_i^2\bigr) - \bigl(\lvert\overrightharpoon{v_a}\rvert^2 - r_a^2\bigr),
	\]
	which determine $p$ as an affine function of $\rho$; substituting back into $\dist{pv_a}^2 = (\rho + r_a)^2$ then leaves a single quadratic equation in $\rho$. The center and radius of each added circle therefore lie in a quadratic extension of the field generated, over the rationals, by the center coordinates and radii of the current configuration, and since the validity test in step (2) and the gap counts in step (3) only compare such quantities, the entire subdivision of Corollary~\ref{cor:bound} can be carried out with straightedge and compass. This linearization is classical. It is precisely how Courant and Robbins \cite[pp.~125--127]{CourantRobbins1996} show that the tangency problem of Apollonius is solvable by straightedge and compass, and an elegant explicit construction, via inversion, is given in the same reference, pp.~161--162. Everything above applies verbatim to the inflated configurations of Sections~\ref{sec:rivers} and~\ref{sec:multirivers}, in which $r_i$ is replaced by $r_i + w$ or $r_i + W_i$. (Our implementation computes the vertices from this quadratic in closed form.)


\section{Optimality of the bound}
\label{sec:optimality}


The bound of Corollary~\ref{cor:bound} cannot be improved. In this section, for arbitrarily large $N$, we construct $N$-gons of touching circles that cannot be subdivided into triangles and quadrilaterals by fewer than $N-4$ added circles (Corollary~\ref{cor:tight}). The witnesses form a family of one simple shape---two half-staggered rows of equal circles closed by two small caps---and the plan is equally simple. Subdividing a member of the family leaves, besides triangles and quadrilaterals, only smaller members of the same family, so the conclusion follows by induction. We begin with a bottleneck estimate at the interlocking ends of two rows.

\begin{lemma}[throat lemma]\label{lem:throat}
	Let $D_i$ and $D_j$ be two disks of radius $4$ whose centers $v_i$ and $v_j$ are at distance $4\sqrt5$, and call the part of the segment $\overline{v_iv_j}$ outside the two disks, of length $4\sqrt5 - 8$, their \emph{throat}. Then every disk avoiding the interiors of $D_i$ and $D_j$ and containing a point of the throat has radius at most $\tfrac12$.
\end{lemma}

\begin{proof}
	Let the disk have center $p$ and radius $\rho$, and let $q$ be a point of the throat inside it. Write $\delta_i = \dist{qv_i}$ and $\delta_j = \dist{qv_j}$, so that $\delta_i, \delta_j \ge 4$ and $\delta_i + \delta_j = 4\sqrt5$. Work in Cartesian coordinates with origin $v_i$ and first axis in the direction of $v_j$, so that $q = (\delta_i, 0)$, and let $\alpha$ denote the first coordinate of $p$. The second coordinates of $\overrightharpoon{p} - \overrightharpoon{v_i}$ and $\overrightharpoon{p} - \overrightharpoon{q}$ coincide, so $\dist{pv_i}^2 - \alpha^2 = \dist{pq}^2 - (\alpha - \delta_i)^2$, that is, $\dist{pv_i}^2 = \dist{pq}^2 + 2\delta_i \alpha - \delta_i^2$. Combining this with the constraints $(4+\rho)^2 \le \dist{pv_i}^2$ and $\dist{pq} \le \rho$ gives $(4+\rho)^2 \le \rho^2 + 2\delta_i\alpha - \delta_i^2$, that is,
	\[
		\alpha \ge \frac{16 + 8\rho}{2\delta_i} + \frac{\delta_i}2,
		\quad\text{and symmetrically}\quad
		4\sqrt5 - \alpha \ge \frac{16+8\rho}{2\delta_j} + \frac{\delta_j}2.
	\]
	Adding and using $1/\delta_i + 1/\delta_j \ge 4/(\delta_i+\delta_j)$,
	\[
		4\sqrt5 \;\ge\; (8 + 4\rho)\left(\frac1{\delta_i} + \frac1{\delta_j}\right) + \frac{\delta_i+\delta_j}{2}
		\;\ge\; \frac{8+4\rho}{\sqrt5} + 2\sqrt5,
	\]
	so $\rho \le \tfrac12$. \qedhere
\end{proof}

Call a polygon of touching circles a \emph{two-row configuration} if its circles are as follows. Two \emph{rows} of radius-$4$ circles, tangent consecutively within each row, have their centers on two parallel lines at distance $8$, and the two rows are offset against each other by $4$---half the center spacing. At each end, the rows terminate in a pair of \emph{extreme} circles, one per row, whose centers we require to be at distance $\sqrt{4^2 + 8^2} = 4\sqrt5$; the rows thus overhang each other by exactly $4$ at each end, and their lengths differ by at most one circle. Finally, at each end sits a \emph{cap}, a circle tangent to the two extreme circles of that end, whose disk avoids their throat. The component of $\operatorname{int} P$ minus the two extreme disks and the throat that contains the cap is called the \emph{end zone} of that end, and we require that the closure of each end zone meet no circle other than its cap and its two extremes.

The end zone is always contained in the triangle spanned by the two extreme centers and the center of the cap. Indeed, the boundary of $P$ near the end consists of the two edges joining the cap center to the two extreme centers, which are covered by their tangent disks (Lemma~\ref{lem:cover}), while the throat lies on the segment joining the extreme centers.

For $k \ge 5$, we write $T_k$ for the special case of a two-row configuration with $k$ circles in total in which both caps have radius exactly $1$. The general family, whose caps are only required to satisfy the conditions above, is what the induction below runs on; $T_k$ is its canonical member, and is determined up to isometry. Indeed, the $k-2$ row circles split into two rows whose lengths differ by at most one, and a cap of radius $1$, being tangent to the two extreme disks of its end, has its center at distance exactly $\sqrt{3^2+4^2} = 5 = 4 + 1$ from the two extreme centers---on the far side of their segment from the rows, the near side being excluded by the end-zone requirement. Placing the row centers at the points $(8i, 0)$ and $(8i+4, 8)$, all the remaining requirements hold---two circles of opposite rows are at distance at least $4\sqrt5 > 8$, and each end-zone triangle is at distance greater than $7$ from every center outside its end, so the end zones indeed meet no other circles. Thus $T_k$ is a $k$-gon of touching circles; Figure~\ref{fig:tight} shows $T_9$.

\begin{figure}[tb]
	\centering
	\begin{tikzpicture}
		\node[anchor=south west, inner sep=0] (img) at (0,0) {\includegraphics[width=0.72\textwidth]{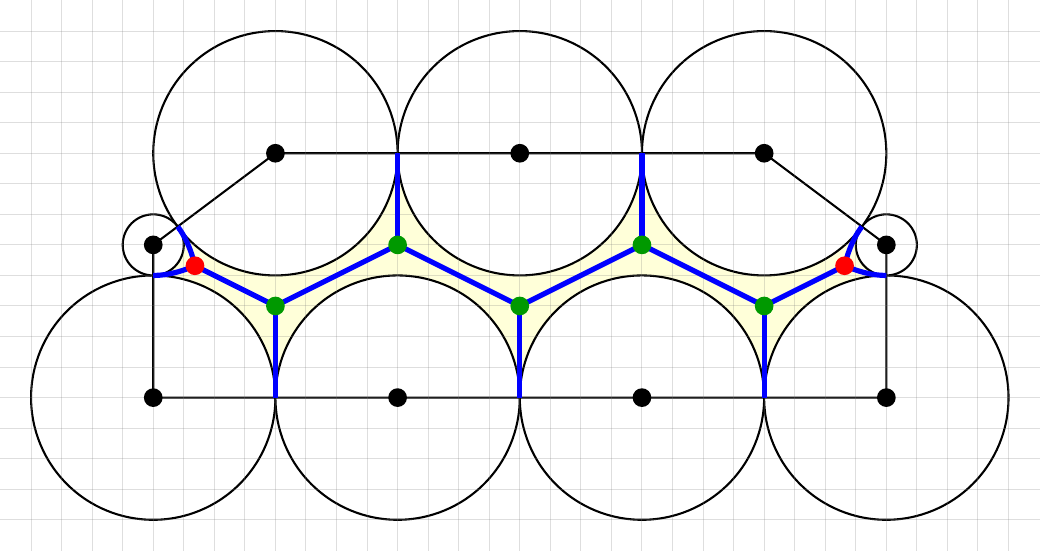}};
		\begin{scope}[x={(img.south east)}, y={(img.north west)}, line width=0.5pt]
			\footnotesize
			\node[left] at (0.078,0.64) {cap};
			\draw (0.078,0.628) -- (0.125,0.592);
			\node at (0.274,0.620) {end zone};
			\draw (0.235,0.588) -- (0.187,0.542);
			\node at (0.191,0.378) {throat};
			\draw (0.196,0.404) -- (0.201,0.478);
		\end{scope}
	\end{tikzpicture}
	\caption{The configuration $T_9$, two half-staggered rows of radius-$4$ circles closed by two radius-$1$ caps, shown with the medial axis of its free region and with the cap, end zone, and throat of its left end labeled. The five green vertices are the integer points $(8i+4,3)$ and $(8i+8,5)$, whose maximal disks have radius exactly $1$. All seven internal vertices have degree $3$ and touch a consecutive pair, and the subdivision at any vertex leaves---besides triangles and quadrilaterals---only two-row configurations again (Proposition~\ref{prop:tworow}).}
	\label{fig:tight}
\end{figure}

\begin{proposition}\label{prop:tworow}
	In a two-row configuration, every admissible circle touches exactly three circles, of which two are consecutive; and every reduced polygon of the resulting subdivision with at least five sides is again a two-row configuration.
\end{proposition}

\begin{proof}
	Consider first a circle that can be added, whose disk meets a closed end zone or a throat. If it meets a throat, then by Lemma~\ref{lem:throat} it has radius at most $\tfrac12$, and its center lies within $1$ of the segment joining the two extreme centers; that segment is at distance greater than $7$ from the center of every circle other than the cap and the two extreme circles, so the disk, at distance greater than $6 > 4 + \tfrac12$ from all such centers, touches none of them. If it avoids the throats, then, being connected and disjoint from the interiors of the extreme disks, it lies inside the end zone, whose closure meets no other circle. In either case the circle can touch only the cap and the two extreme circles of that end. If it is admissible, it therefore touches exactly these three circles, so $k = 3$, and two of the touched circles are consecutive---indeed the cap is consecutive to both extremes.

	Otherwise the disk avoids the closed end zones and the throats, and hence touches only circles of the rows. Place the configuration so that the centers of one row are at points $(8i, 0)$ and those of the other at points $(8j+4, 8)$, for integers $i$ and $j$. We claim that, if the circle is admissible, it touches exactly two consecutive circles of one row and one circle of the other---so $k = 3$ with a consecutive pair, as required. This follows from two features of the rows. First, a free disk between them is too small to reach two non-consecutive circles of a row; second, the half-stagger puts the candidate centers for the two rows on disjoint vertical lines, ruling out two touches per row. Quantitatively, consider a free disk centered at $(x,y)$ with $x$ in the common span of the rows---the interval of abscissas covered by both---so that each row, its centers spaced $8$ apart, has a center within horizontal distance $4$ of $x$. The disk then has radius at most
	\[
		\min\left(\sqrt{16 + y^2},\, \sqrt{16 + (8-y)^2}\right) - 4 \;\le\; \sqrt{32} - 4 \;<\; 4,
	\]
	the two entries of the minimum bounding its distances to the nearest center of the bottom and of the top row, and the minimum being largest at $y = 4$; while a disk touching two non-consecutive circles of the same row, whose centers are at distance at least $16$, has radius at least $4$. Such a circle therefore touches at most two circles per row, and only consecutive ones. Moreover a disk touching two consecutive bottom circles is centered on a line $x = 8i+4$, and one touching two consecutive top circles on a line $x = 8j+8$; these lines are disjoint, so no circle touches two circles of each row. An admissible circle, touching at least three circles in total, therefore touches exactly two consecutive circles of one row and one circle of the other. Concretely, in $T_k$, placed as in Figure~\ref{fig:tight}, the resulting vertices are the integer points $(8i+4, 3)$ and $(8i+8, 5)$, whose maximal disks have radius exactly $1$. Each is at distance $5$ from its three touched circles, and at distance at least $\sqrt{64+25} = \sqrt{89} > 5+1$ from the centers of all other circles of radius $4$.

	It remains to examine the reduced polygons, the added circle being of one of the two kinds just identified. A circle added over the rows splits the configuration, by Lemma~\ref{lem:split}, into a triangle and one or two shorter configurations of the very same two-row form, in which the added radius-$1$ disk becomes a cap. For instance, a disk added at $(4, 3)$, as a cap for the extreme circles at $(8,0)$ and $(4,8)$, is again tangent to both at distance $5$; its center lies at distance $\sqrt5$ from the segment joining the two extreme centers, on the end-zone side, so the disk of radius $1$ avoids their throat---the same integer data as the original caps, reflected. A circle added at an end cuts off two triangles and leaves an $N$-gon with the same rows and extremes, in which the added disk---contained in the same end zone and tangent to the two extreme circles---becomes the new, smaller cap. In both cases each new cap is tangent to its two extreme circles and its disk lies in the corresponding end zone, which meets no other circle, so every reduced polygon with at least five sides satisfies the definition of a two-row configuration.
\end{proof}

\begin{corollary}\label{cor:tight}
	For arbitrarily large $N$ there exist $N$-gons of touching circles that cannot be subdivided into triangles and quadrilaterals by fewer than $N-4$ added circles.
\end{corollary}

\begin{proof}
	Since $T_N$ is a two-row configuration with $N$ sides for every $N \ge 5$, it suffices to show that subdividing a two-row configuration with $N$ sides into triangles and quadrilaterals requires at least $N-4$ added circles. We argue by induction on the total number of circles added. Since $N \ge 5$, at least one circle is added. By Proposition~\ref{prop:tworow}, the first one touches exactly three circles with a consecutive pair, so its gaps are $1, g_2, g_3$ with $g_2 + g_3 = N-1$, and it subdivides the polygon into a triangle together with reduced polygons of $g_2 + 2$ and $g_3 + 2$ sides. Every later circle is added inside one of these reduced polygons, and a reduced polygon with $g_j + 2 \ge 5$ sides is again a two-row configuration (Proposition~\ref{prop:tworow}), whose complete subdivision uses at least $g_j - 2$ added circles by the induction hypothesis. The total number of added circles is therefore at least $1 + \max(0, g_2 - 2) + \max(0, g_3 - 2)$. One case to watch for is $g_j = 1$, arising when the first circle is added at an end. The large reduced polygon then has $N$ sides again---the same rows with a smaller cap---so the addition was a wasted move, and this is why the induction runs on the number of added circles rather than on $N$. The arithmetic confirms the picture. When $g_2, g_3 \ge 2$ the total is at least $1 + (g_2 - 2) + (g_3 - 2) = N - 4$; and when some $g_j = 1$, the other gap is $N - 2$, and the total is at least $1 + (N - 2) - 2 = N-3$. In either case at least $N-4$ circles are added.
\end{proof}


\section{Polygons with rivers}
\label{sec:rivers}

In the circle/river method, a \emph{river} is a curved strip of constant width representing a segment shared by several flaps; within an axial polygon, a river enters through one gap between consecutive circles and leaves through another, separating the circles of the polygon into two groups. We now show that Theorem~\ref{thm:main} persists in the presence of a river.

\begin{definition}\label{def:riverring}
	An \emph{$N$-gon of touching circles with a river of width $w>0$} consists of a simple polygon $P$ with vertices $v_1,\dots,v_N$, circles $C_i$ centered at $v_i$ with radii $r_i > 0$, and two distinguished non-adjacent edge indices $a \ne b$, partitioning the circles into the two \emph{banks} $A = \{C_{a+1},\dots,C_b\}$ and $B = \{C_{b+1},\dots,C_a\}$, such that
	\begin{enumerate}
		\item $\dist{v_iv_{i+1}} = r_i + r_{i+1}$ for $i \ne a,b$, while $\dist{v_iv_{i+1}} = r_i + r_{i+1} + w$ for $i \in \{a,b\}$; and
		\item $\dist{v_iv_j} > r_i + r_j$ for non-consecutive $C_i, C_j$ on the same bank, and $\dist{v_iv_j} > r_i + r_j + w$ for non-consecutive $C_i, C_j$ on opposite banks.
	\end{enumerate}
	In the underlying design, the river enters $P$ through the gap on the edge $\overline{v_av_{a+1}}$, separates the flaps of $A$ from those of $B$, and leaves through the gap on $\overline{v_bv_{b+1}}$.
\end{definition}

Figure~\ref{fig:river}(a) shows an octagon of touching circles with a river. We assume throughout this section that the circle is to be added on the bank $A$; this loses no generality, as exchanging the indices $a$ and $b$ exchanges the two banks. A circle $C$ with center $p \in \operatorname{int} P$ and radius $\rho$ added on the bank $A$ must not overlap the circles of $A$, and must in addition leave room for the river to pass between itself and every circle of $B$; that is,
\begin{equation}\label{eq:riveradd}
	d_i(p) \ge \rho \quad \text{for } C_i \in A,
	\qquad
	d_i(p) \ge \rho + w \quad \text{for } C_i \in B.
\end{equation}
When equality holds we say that $C$ \emph{touches} $C_i$ (for $C_i \in A$), or touches it \emph{across the river} (for $C_i \in B$); in the latter case the segment $\overline{pv_i}$ has length $\rho + w + r_i$, matching the length of an axial-polygon edge crossed by the river. Subdividing along the segments from $p$ to the touched centers then decomposes $P$ exactly as in Lemma~\ref{lem:split}, with the river now threading through those sub-polygons whose boundary separates the two banks; the count of \eqref{eq:gap} is unchanged.

\begin{figure}[tbp]
	\centering
	\subfloat[]{\includegraphics[width=0.5\textwidth]{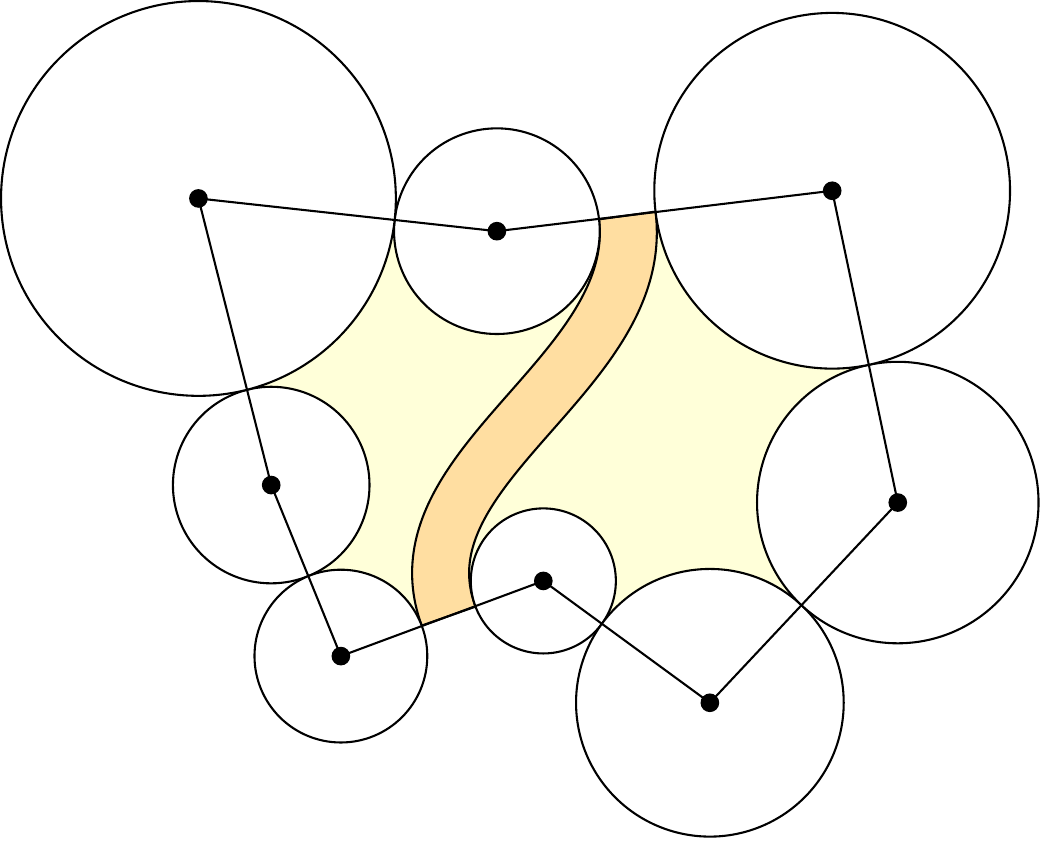}}\\
	\subfloat[]{\includegraphics[width=0.46\textwidth]{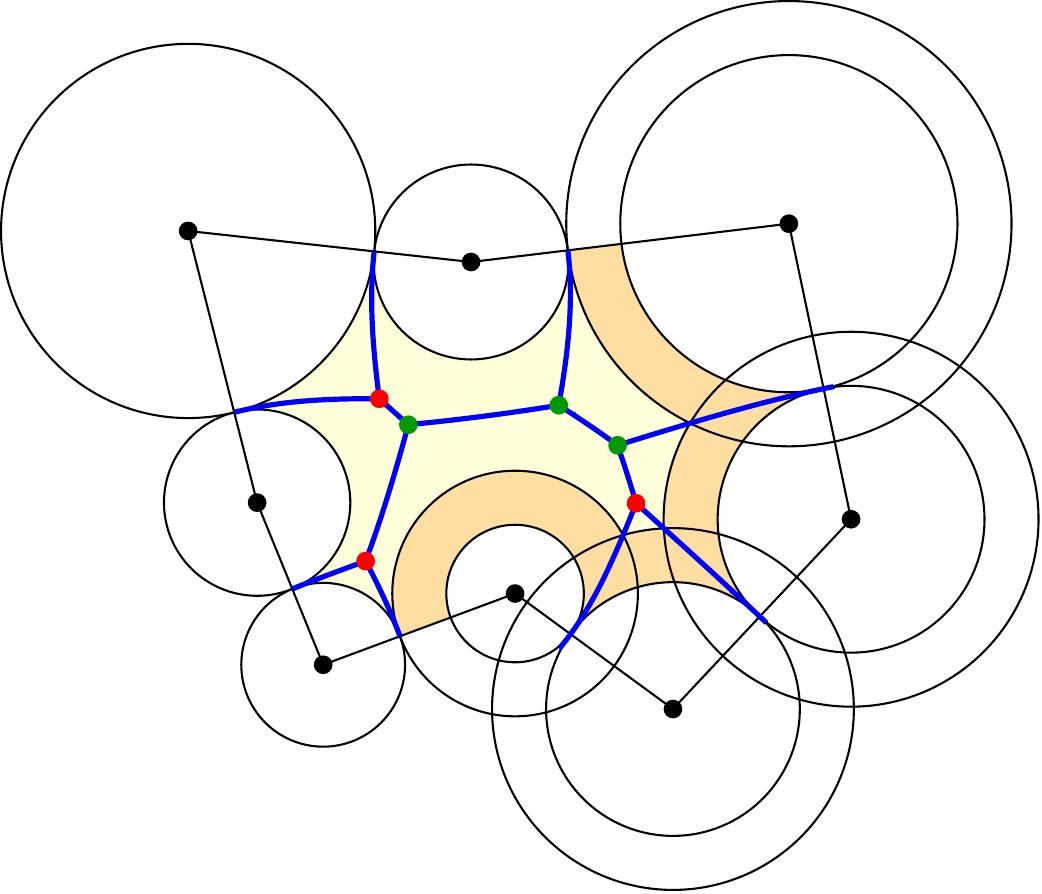}}\hspace{0.04\textwidth}
	\subfloat[]{\includegraphics[width=0.46\textwidth]{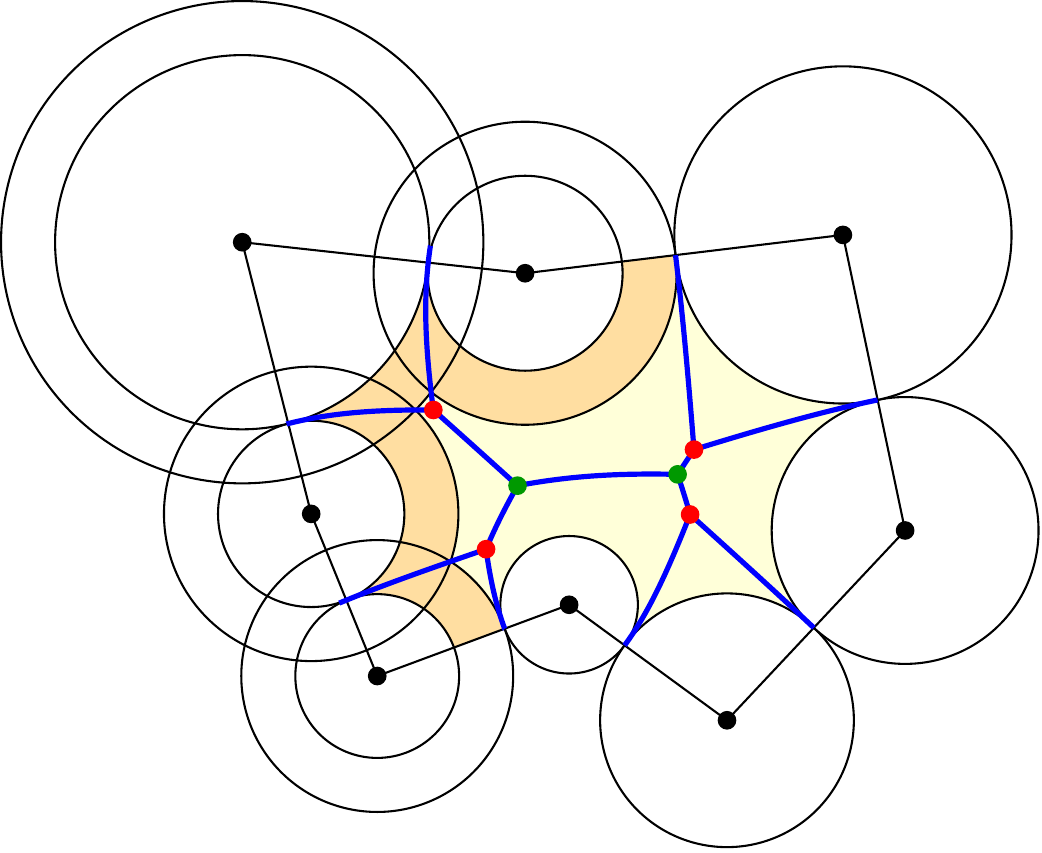}}
	\caption{(a) An octagon of touching circles with a river of width $w$ (shaded orange), drawn as the designer conceives it---a strip of width $w$ entering and leaving the polygon through the two gaps. (b), (c) The medial axes (blue) of the free regions of the two expanded configurations of the same example. In (b) the circles of one bank, and in (c) those of the other, are enlarged by $w$ (enlarged circles are drawn with double outlines), and the collar between the original and the enlarged circles---the region along which the river is rerouted in the proof of Theorem~\ref{thm:river}---is shaded in the same orange. Consecutive enlarged circles overlap, and their boundary arcs meet transversally instead of at cusps; the medial axis is nevertheless a tree with $N = 8$ leaves in both cases. Internal vertices are colored green or red as in Figure~\ref{fig:tree}. The two choices of bank yield two different reductions of the same configuration.}
	\label{fig:river}
\end{figure}

The device of the proof is to absorb the river into the circles of the bank $B$, and to perform the absorption \emph{continuously}. For $t \in [0,w]$ define
\begin{equation}
	\tilde r_i(t) = r_i \quad \text{for } C_i \in A,
	\qquad
	\tilde r_i(t) = r_i + t \quad \text{for } C_i \in B,
\end{equation}
write $\tilde C_i(t)$, $\tilde D_i(t)$ for the corresponding circles and disks, and let
\[
	\FR(t) \;=\; \overline{P \setminus (\tilde D_1(t) \cup \cdots \cup \tilde D_N(t))}
\]
be the free region at time $t$, which shrinks as $t$ grows. At $t = w$ we obtain the \emph{expanded configuration}. By \eqref{eq:riveradd}, a circle can be added on the bank $A$ exactly when it can be added to the expanded configuration in the sense of Definition~\ref{def:add}. The expanded configuration is no longer a polygon of \emph{touching} circles in the strict sense. Consecutive circles within the bank $B$ overlap, their boundaries crossing transversally (Figure~\ref{fig:river}), the two gap pairs become tangent, and in degenerate configurations an expanded disk may even swallow the free arcs of other circles of $B$ entirely (Figure~\ref{fig:swallow}). Rather than excluding such degenerations, we track them along the shrinking process. For $t < w$, the boundary of $\FR(t)$ contains two segments of length $w - t$ on the gap edges, which we call the \emph{mouths}. They are the openings through which the river enters and leaves the polygon, and at $t = w$ they close into two tangency cusps.

\begin{figure}[tb]
	\centering
	\subfloat[]{\includegraphics[width=0.44\textwidth]{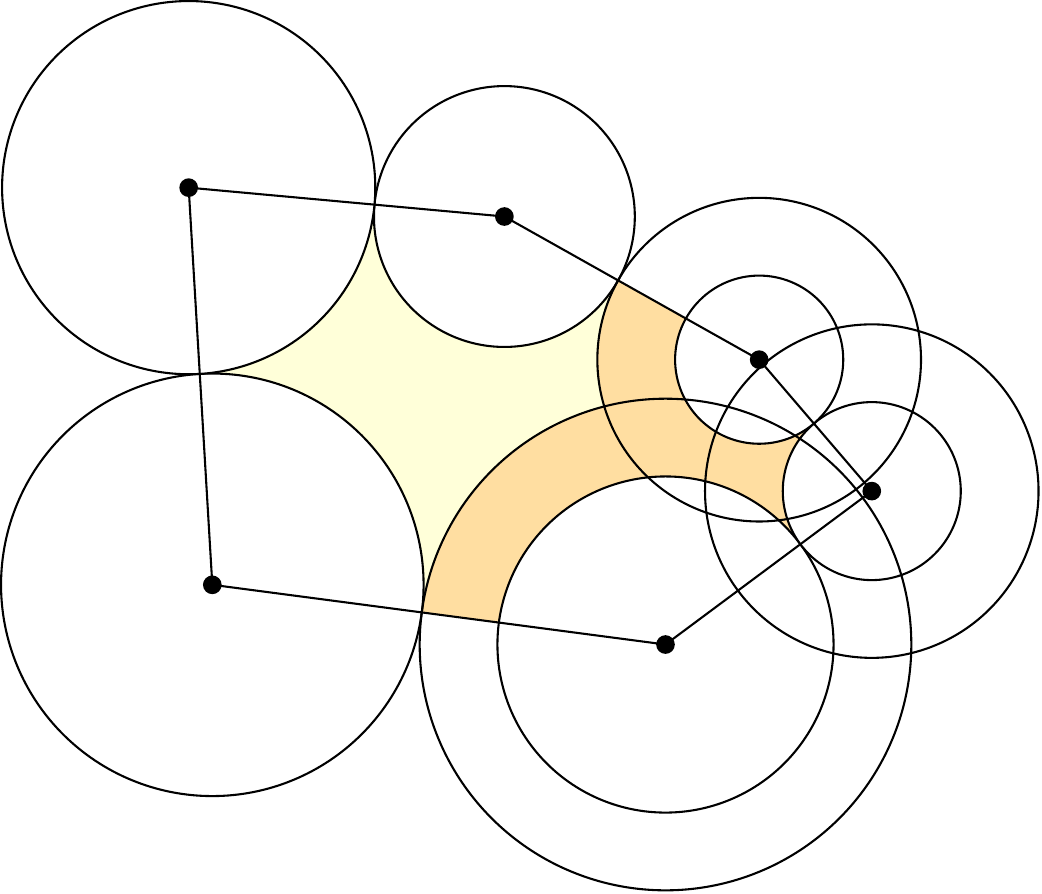}}\hspace{0.04\textwidth}
	\subfloat[]{\includegraphics[width=0.44\textwidth]{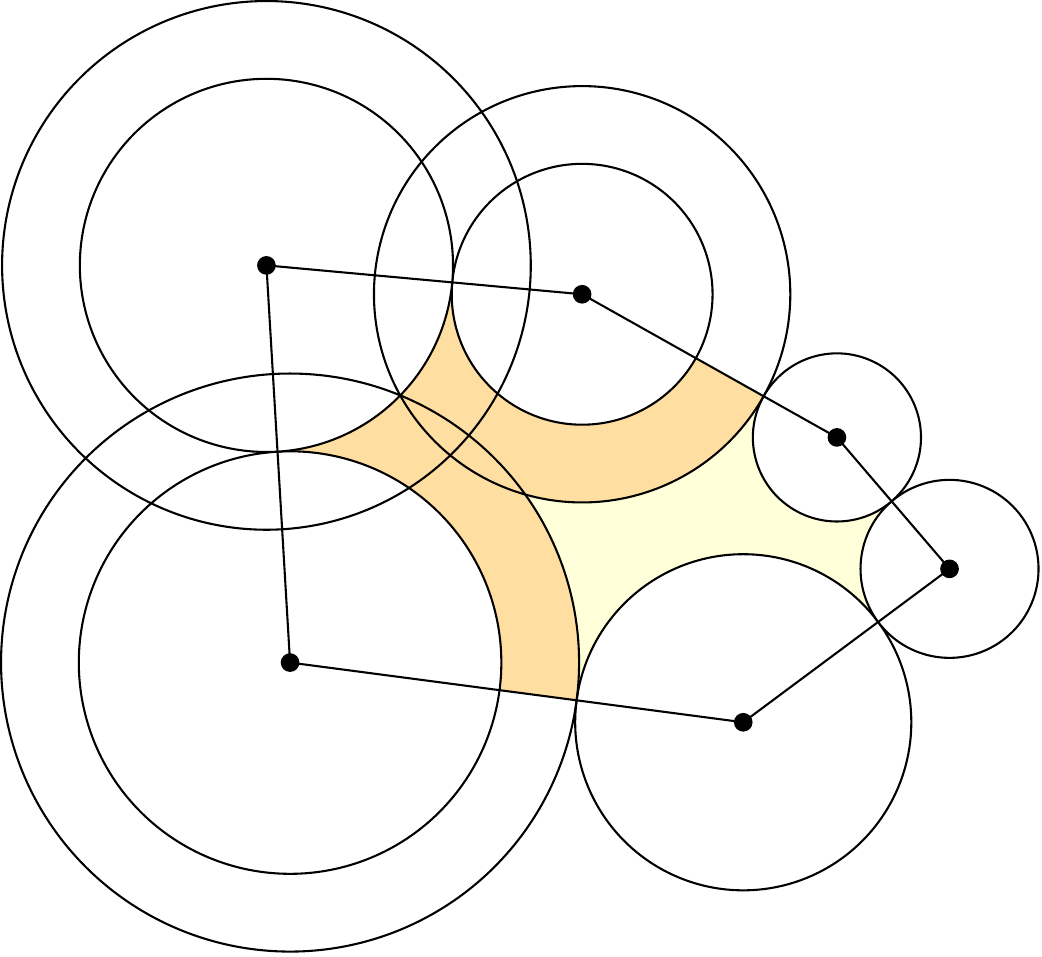}}
	\caption{A hexagon of touching circles with a river, expanded on either bank (enlarged circles drawn with double outlines, the collar shaded orange as in Figure~\ref{fig:river}). In both (a) and (b) the expansion is degenerate: the middle circle of the expanded bank is swallowed by its two enlarged neighbors, its free arc disappearing entirely, so that the boundary of the free region runs from one neighbor directly to the other.}
	\label{fig:swallow}
\end{figure}

Two consequences of Definition~\ref{def:riverring} confine all degeneration strictly to the interior of the expanded bank.

\begin{lemma}[bank integrity]\label{lem:integrity}
	For all $t \in [0,w]$ and every $C_j \in B$:
	\begin{enumerate}
		\item\label{it:crossbank} $\tilde D_j(t)$ meets no disk of the bank $A$, except that the two gap pairs become tangent at $t = w$; in particular it contains no center of a circle of $A$ and meets no edge of $P$ joining two circles of $A$;
		\item\label{it:mouth} the interior of $\tilde D_j(t)$ contains no point of either mouth.
	\end{enumerate}
\end{lemma}

\begin{proof}
	\ref{it:crossbank} For $C_m \in A$ non-consecutive to $C_j$, Definition~\ref{def:riverring} gives $\dist{v_jv_m} > r_j + r_m + w \ge \tilde r_j(t) + r_m$; for the gap pairs, $\dist{v_jv_m} = r_j + r_m + w \ge \tilde r_j(t) + r_m$, with equality only at $t = w$. An edge joining two circles of $A$ is covered by their disks, so meeting it would mean meeting a disk of $A$.

	\ref{it:mouth} If $C_j$ is the bank-$B$ circle of the mouth's own gap edge, the mouth consists by definition of points at distance at least $\tilde r_j(t)$ from $v_j$. Otherwise, let $C_m \in A$ be the bank-$A$ circle of that gap edge; every point $q$ on the mouth satisfies $\dist{qv_m} \le r_m + (w - t)$, so $q \in \tilde D_j(t)$ would give
	\[
		\dist{v_jv_m} \;\le\; (r_j + t) + (r_m + (w-t)) \;=\; r_j + r_m + w,
	\]
	contradicting Definition~\ref{def:riverring} for the non-consecutive cross-bank pair $C_j$, $C_m$.
\end{proof}

The next lemma is the structural heart of this section---no matter how badly the expanded bank degenerates, the boundary of the free region can never scramble the circles.

\begin{lemma}[shrinking lemma]\label{lem:shrink}
	For every $t \in [0,w]$, each connected component of $\FR(t)$ is bounded by circular arcs of distinct circles---at most one arc per circle---occurring along the boundary in their cyclic index order, interleaved in the correct position with at most two mouth segments.
\end{lemma}

\begin{proof}
	At $t = 0$ the circles are pairwise tangent or disjoint with the strict clearances of Definition~\ref{def:riverring}, and the proofs of Lemmas~\ref{lem:cover} and~\ref{lem:domain} apply verbatim except on the two gap edges, where the uncovered mouths appear. Thus $\FR(0)$ is a single component bounded by the $N$ arcs in cyclic order and the two mouths. It remains to show that this structure can only degenerate, never reorganize, as $t$ grows.

	First, boundary material only disappears, never appears. Fix $C_j \in B$ and a unit vector $u$, and follow the point $p(t)$ with $\overrightharpoon{p}(t) = \overrightharpoon{v_j} + \tilde r_j(t)\,u$ of $\tilde C_j(t)$. Call the direction $u$ \emph{dead} at time $t$ if $p(t)$ lies in the interior of another expanded disk or outside $P$. Death is permanent. If $p(t_0) \in \tilde D_m(t_0)$ for some $C_m \in B$, then for $t \ge t_0$,
	\[
		\dist{p(t)\,v_m} \;\le\; \dist{p(t_0)\,v_m} + (t - t_0) \;\le\; \tilde r_m(t_0) + (t - t_0) \;=\; \tilde r_m(t),
	\]
	since $p$ moves with unit speed while $\tilde D_m$ grows with unit speed; disks of $A$ never meet $\tilde C_j(t)$ at all, by Lemma~\ref{lem:integrity}\ref{it:crossbank}; and if $p(t_0)$ exits $P$, then at the first later time at which $p$ returns to $\partial P$, the return point is neither on a vertex or edge of the bank $A$ (Lemma~\ref{lem:integrity}\ref{it:crossbank}) nor on a mouth (Lemma~\ref{lem:integrity}\ref{it:mouth}), so it lies on a part of an edge covered by the expanded disk of a circle of $B$---and such coverage is permanent by the first computation. The points of the circles of $A$, finally, are never covered by anything. Hence every boundary arc---viewed as the set of directions $u$ it spans, its radius never decreasing---can only shrink, split, or vanish, and no circle ever contributes new boundary material.

	Second, the combinatorial structure changes only at finitely many critical times, namely when an arc or mouth degenerates to a point, or when two expanded circles become tangent at a point of the closure of the free region (tangencies occurring inside covered territory or outside $P$ are invisible); by Lemma~\ref{lem:integrity}, such a tangency can only involve two circles of $B$, or a gap pair at the final time $t = w$, where a mouth closes into a cusp. When an arc or mouth degenerates, one entry is deleted from the boundary sequence of its component, which preserves the asserted structure. When two circles become tangent at a free point, the free region is pinched there, and the two local sides of the pinch end up in different components. Suppose not. Then some path inside the free region joins the two sides, and, closed up through the contact point, it yields a Jordan curve $J$ contained in the interior of $P$. Near the contact point, the two tangent disks lie on opposite local sides of $J$, so $J$ encloses one of the two disks---and with it the disk's center, which is a vertex of $P$. On the other hand, $\partial P$ is connected, avoids $J$, and reaches points arbitrarily far away, so all of $\partial P$ lies outside $J$; in particular every vertex of $P$ does, a contradiction. So the pinch splits the component in two, and each new component inherits a contiguous run of the old boundary sequence, together with the new corner. In all cases the invariant survives, and the lemma follows by induction over the critical times.
\end{proof}

\begin{lemma}\label{lem:main}
	At $t = w$, some component $\FR_\ast$ of $\FR(w)$ carries on its boundary the arcs of all circles of $A$, the two mouth cusps, and arcs of the two circles of $B$ adjacent to the gap edges.
\end{lemma}

\begin{proof}
	By Lemma~\ref{lem:integrity}, no expanded disk meets any point of an arc $\mathcal{A}_m$ of a circle $C_m \in A$, so the free side of the whole $A$-chain survives at $t = w$. The free points at sufficiently small positive distance from the arcs of the chain form a connected set---it passes through the interior wedges at the tangency cusps of the chain---and hence lie in a single component $\FR_\ast$. The boundary of $\FR_\ast$ then contains every $\mathcal{A}_m$, the two mouth cusps at the two ends of the chain, and, beyond each mouth cusp, an initial piece of the arc of the adjacent expanded circle bounding the same wedge.
\end{proof}

\begin{theorem}\label{thm:river}
	For any $N$-gon of touching circles with a river, with $N > 4$, it is possible to add a circle on either bank, touching at least three of the circles---directly or across the river---so that the largest reduced polygon has at most $N-1$ sides.
\end{theorem}

\begin{proof}
	Let $\FR_\ast$ be the component of Lemma~\ref{lem:main}, and let $\MA = \MA(\FR_\ast)$. By Lemma~\ref{lem:shrink}, $\partial\FR_\ast$ consists of arcs of distinct circles in cyclic index order, meeting at corners which are tangency cusps or transversal crossings; at a crossing the free region occupies the sector outside both disks, whose angle is less than $\pi$, so all corners are sharp and $\FR_\ast$ is a simply connected domain of the class of Theorem~\ref{thm:ccm}. The analysis of Lemma~\ref{lem:tree} applies to $\MA$ with one modification. In the proof of Lemma~\ref{lem:tree}\ref{it:internal} we showed that no contact point of a maximal disk is a corner of the boundary; there every corner was a tangency cusp, whereas here a corner may also be a transversal crossing, which requires an argument of its own. If a maximal disk centered at $p$ made contact at a transversal crossing $x$ of two expanded circles, then $x$ would be the point of each of the two circles nearest to $p$, forcing $p$ to lie on the ray through $x$ from each of the two centers; these rays are distinct, so $p = x$ and $\rho = 0$, a contradiction. Thus $\MA$ is a finite tree whose leaves are the corners of $\partial\FR_\ast$; every non-leaf vertex, of degree $k$, is the center of a disk touching exactly $k$ of the expanded circles, each of which contributes an arc to $\partial\FR_\ast$, since the contact points of a maximal disk of $\FR_\ast$ lie on $\partial\FR_\ast$; and the leaves of $\MA$ are distributed among the branches at such a vertex according to the boundary stretches, as in Lemma~\ref{lem:tree}\ref{it:branches}. By \eqref{eq:riveradd}, the circle bounding such a maximal disk can be added on the bank $A$, touching the corresponding circles directly (bank $A$) or across the river (bank $B$).

	Call a circle \emph{hidden} if it contributes no arc to $\partial\FR_\ast$. By Lemma~\ref{lem:main}, the circles of $A$ and the two circles of $B$ next to the gap edges are never hidden, so at least $|A| + 2 \ge 4$ circles contribute arcs, the gap edges being non-adjacent.

	\emph{Case 1: at least five circles contribute arcs.} The leaves of $\MA$ are the corners of $\partial\FR_\ast$, one between each pair of consecutive contributing arcs, so their number $L \ge 5$ equals the number of contributing circles. Consider a non-leaf vertex and one of its branches, running between consecutively touched circles $C_{i_j}$ and $C_{i_{j+1}}$. The branch contains one leaf for each contributing circle strictly between the two, plus one, while the gap $g_j$ counts one for \emph{every} circle strictly between the two, plus one; the gap therefore exceeds the leaf count of its branch by the number of hidden circles strictly between, and in particular by at most $N - L$. At the vertex provided by the tree lemma (Lemma~\ref{lem:treelemma}), every branch contains at most $L - 3$ leaves, so every gap is at most $(L-3) + (N-L) = N-3$. (When no circle is hidden, this is the proof of Theorem~\ref{thm:main} verbatim.)

	\emph{Case 2: exactly four circles contribute arcs.} The tree lemma, which requires at least five leaves, has nothing to offer here, and we locate the vertex by hand. Since the circles of $A$ and the two circles of $B$ next to the gap edges all contribute, four contributors force $|A| = 2$ and leave no contributor besides these four; the hidden circles thus form the single run strictly between the two contributing circles of $B$, say $C_p$ and $C_q$ with $q - p = |B| - 1 = N - 3$. By Lemma~\ref{lem:shrink}, the arcs of $C_p$ and $C_q$ are adjacent on $\partial\FR_\ast$; the edge of $\MA$ emanating from the corner between them terminates at a non-leaf vertex $I$, the tree $\MA$ having four leaves and hence not being a single edge, and $I$ touches $C_p$, $C_q$, and one or both circles of $A$. The gaps of $I$ are therefore $q - p = N-3$ together with two or three further gaps summing to $3$, and since $N \ge 5$, all of them are at most $N-3$.

	In either case we obtain a vertex all of whose gaps satisfy $g_j \le N-3$; add the circle bounding its maximal disk. The river remains realizable. Its course may be taken to run along the strip of width $w$ swept along the free-side boundary of the union $\bigcup_{C_i \in B} D_i$, from the gap on $\overline{v_av_{a+1}}$ to the gap on $\overline{v_bv_{b+1}}$; this strip is disjoint from the disks of $A$ and from the added circle, since both keep distance at least $w$ from every disk of $B$, and it lies inside $P$, squeezed between $\bigcup_{C_i \in B} D_i$ and the free region. Finally, subdividing along the segments from the vertex to the centers of the touched circles decomposes $P$, exactly as in Lemma~\ref{lem:split}, into polygons with $g_j + 2 \le N-1$ sides; each is again a polygon of touching circles, either plain (if its boundary chain stays within one bank together with the added circle) or carrying the river across it (if its chain contains circles of both banks, or a tangency across the river).
\end{proof}

\begin{remark}
	The two banks yield, in general, two genuinely different reductions of the same configuration, as shown in Figure~\ref{fig:river}. In a design context this freedom can be used to place the added flap on whichever side of the river is more convenient.
\end{remark}


\section{Polygons with several rivers}
\label{sec:multirivers}


A circle-packed design may send several rivers through the same axial polygon. Since the courses of distinct rivers are disjoint strips, the separations they induce on the circles do not cross. The rivers thus form a \emph{laminar} family, reflecting the tree structure of the underlying uniaxial base.

\begin{definition}\label{def:multiring}
	An \emph{$N$-gon of touching circles with rivers} consists of
	\begin{itemize}
		\item a simple polygon $P$ with vertices $v_1,\dots,v_N$;
		\item circles $C_i$ centered at $v_i$ with radii $r_i>0$;
		\item a finite, possibly empty set of \emph{markers}---formal points, each placed between two cyclically consecutive circles and regarded as lying on the edge of $P$ that joins their centers, so that circles and markers together form a single cyclic sequence; and
		\item a finite nonempty collection $\mathcal R$ of \emph{rivers}.
	\end{itemize}
	Each river $R \in \mathcal R$ carries a width $w_R>0$ and a partition of the circles and markers into two \emph{banks}, each a cyclically consecutive stretch of the sequence containing at least two circles; the river \emph{separates} two circles or markers if they lie on different banks of $R$, and \emph{crosses} the two edges of $P$ on which the sequence passes from one bank to the other. We require that no two rivers cross, in the sense that for any $R \ne R'$, some bank of $R$ is properly contained in some bank of $R'$; taking complements, the other bank of $R'$ is then properly contained in the other bank of $R$, so the containments come in a complementary pair and the condition is symmetric in $R$ and $R'$. Writing
	\begin{equation}
		\sigma_{ij} \;=\; \sum\,\bigl\{\, w_R : R \in \mathcal R \text{ separates } C_i \text{ from } C_j \,\bigr\}
	\end{equation}
	for the total width separating two circles, and $w_i = \sigma_{i,i+1}$ for the total width crossing the edge $\overline{v_iv_{i+1}}$, we require further that
	\begin{enumerate}
		\item\label{it:multitangent} $\dist{v_iv_{i+1}} = r_i + r_{i+1} + w_i$ for all $i$; and
		\item\label{it:multidisjoint} $\dist{v_iv_j} > r_i + r_j + \sigma_{ij}$ for all non-consecutive pairs $i, j$.
	\end{enumerate}
\end{definition}

Definition~\ref{def:riverring} is the case of a single river and no markers. Note that $\sigma$ satisfies the triangle inequality $\sigma_{ij} \le \sigma_{ik} + \sigma_{kj}$, since a river separating $C_i$ from $C_j$ separates $C_k$ from one of them.

The markers carry almost no geometry---an edge but no exact position on it, no width, and no mention in the two conditions. Their role is to keep parallel rivers apart. Two rivers running side by side with no circle between them would carry the same partition of the circles, and no partition of the circles alone can record which of the two strips runs first; the proper containments of Definition~\ref{def:multiring} forbid such a pair, and a marker placed between the two rivers---on one of the two edges that both of them cross---restores the proper containments and records the order. This costs no generality, as a design provides such a point in every gap between neighboring strips (Figure~\ref{fig:multiriver}), and nothing prevents marking them all.

\begin{figure}[tb]
	\centering
	\includegraphics[width=0.62\textwidth]{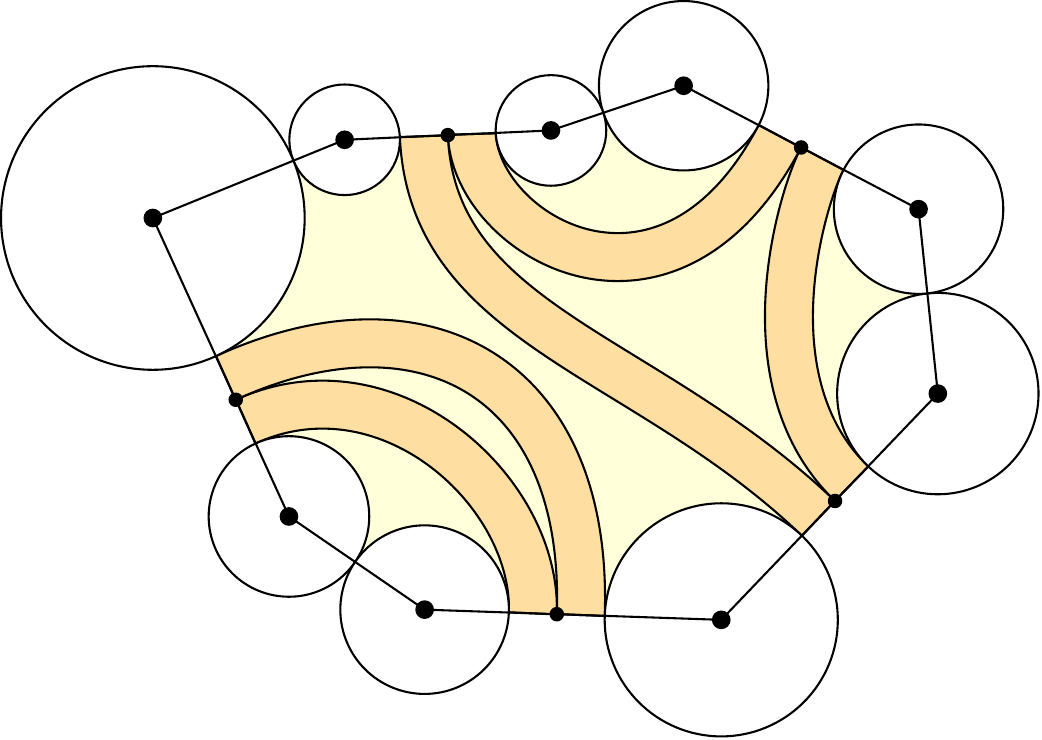}
	\caption{A $9$-gon of touching circles with five rivers, drawn as a designer would conceive them. The two parallel rivers at the lower left carry the same partition of the circles, and only the markers between them tell the two apart. The small black dots are the markers, one in each gap between neighboring strips on an edge. Of the remaining three rivers, the one around the right group of two circles and the one around the upper group have disjoint banks, and both banks are properly contained in the bank of the long river enclosing them both. The rivers cut the polygon into six regions, one for each face (Lemma~\ref{lem:facecount}); the region bounded by the latter three rivers contains no circle, and its three markers alone locate it in the cyclic sequence.}
	\label{fig:multiriver}
\end{figure}

The rivers divide the polygon into regions, and the added circle will be sought in one of them. A region is recorded combinatorially by naming, for each river $R$, the bank $B_R$ on its far side, as seen from the region. Accordingly, a \emph{face} of $\mathcal R$ is a choice of one bank $B_R$ of each river subject to the requirement that the chosen banks be pairwise nested or disjoint---a requirement most choices fail. For a region containing a circle or a marker, the choice is forced---$B_R$ is the bank not containing that member---and the chosen banks are automatically pairwise nested or disjoint, a common member outside two banks ruling out the crosswise containments of laminarity. The combinatorial formulation keeps track, in addition, of regions that contain neither, and frees us from carrying the courses of the rivers, which are not part of our data; the faces are exactly the regions of any drawing of the courses, as Lemma~\ref{lem:facecount} below confirms.

Call a bank $F_0$ \emph{innermost} if no river separates two of its members. Any bank that contains no other bank is innermost, since a river separating two of its members would, by laminarity, have a bank properly inside it.

\begin{lemma}\label{lem:faceexist}
	Choosing, for each river, its bank disjoint from an innermost bank $F_0$ yields a face of $\mathcal R$. In particular, faces exist.
\end{lemma}

\begin{proof}
	The choice is well defined. Indeed, no river separates two members of $F_0$, so each river holds $F_0$ within a single bank, and its other bank is the one disjoint from $F_0$. Now consider the chosen banks of two rivers. Both avoid the members of $F_0$, so the two banks do not cover all the circles and markers, which rules out the crosswise containments of laminarity---the complement of one chosen bank properly inside the other; the containment that laminarity does provide then makes the two banks nested or disjoint.
\end{proof}

In the language of the circle/river method, the faces are the regions of the design visible within the polygon, and the innermost banks give the regions at the leaves of its tree; but a face may equally sit deep in the design. This reading of the faces as the regions is exact, and we now make it precise. A \emph{drawing} of $\mathcal R$ places each marker at a point in the interior of its edge, and the rivers along pairwise disjoint strips---a strip being a closed topological disk, disjoint from the circles and the placed markers, which meets $\partial P$ in two arcs, one in the interior of each of the two edges its river crosses, and whose removal cuts $P$ into two pieces, each containing one bank, circles and markers alike. The \emph{regions} of the drawing are the connected components of the complement of the strips in $P$.

\begin{lemma}\label{lem:facecount}
	In any drawing of $\mathcal R$, sending a region to the banks beyond its strips---for each river, the bank cut off from the region by the strip of that river---is a bijection between the regions of the drawing and the faces of $\mathcal R$.
\end{lemma}

\begin{proof}
	Consider first the regions. Draw the strips one at a time. Each strip is connected and disjoint from the strips drawn before it, so it lies in a single region of those, a topological disk whose boundary the strip meets in its two arcs on $\partial P$; the strip therefore cuts that region into exactly two regions, again disks. Hence the regions number $|\mathcal R| + 1$ once all the strips are drawn, and any two distinct regions lie on opposite sides of at least one strip---the one whose drawing first put them in different regions.

	Next we check that the assignment of banks lands in the faces, and is injective. Fix a region and, for each river $R$, write $P_R$ for the piece of $P$ beyond the strip of $R$, so that the assigned bank $B_R$ is the one inside $P_R$. Given two rivers $R \ne R'$, the strip of $R'$ is connected and disjoint from the strip of $R$, so it lies in $P_R$ or in the other piece. If it lies in $P_R$, then the other piece is connected, avoids the strip of $R'$, and contains the region, so it lies in the piece of $P$ on the region's side of $R'$; passing to complements, $P_{R'}$ is covered by $P_R$ together with the strip of $R$, and $B_{R'} \subseteq B_R$ since circles and markers avoid strips. If instead the strip of $R'$ avoids $P_R$, then $P_R$, being connected, lies either in $P_{R'}$, giving $B_R \subseteq B_{R'}$, or in the piece on the region's side of $R'$, giving $B_R \cap B_{R'} = \emptyset$. The assigned banks are thus pairwise nested or disjoint, and the assignment is a face. It is also injective, since two regions on opposite sides of the strip of some river are assigned the two different banks of that river.

	It remains to show that the faces, too, number exactly $|\mathcal R| + 1$; this part is purely combinatorial and makes no reference to the drawing. We argue by induction on the number of rivers, a single river having exactly two faces. Suppose $|\mathcal R| \ge 2$. Among all banks of all rivers, let $X_0$ be one minimal with respect to inclusion, let $R_0$ be its river, and put $\mathcal R' = \mathcal R \setminus \{R_0\}$. Discarding the bank chosen for $R_0$ turns every face of $\mathcal R$ into a face of $\mathcal R'$, so it suffices to count, for each face of $\mathcal R'$, its extensions back to $\mathcal R$.

	Consider a river $R \ne R_0$. Neither bank of $R$ is contained in $X_0$---a proper containment would contradict the minimality of $X_0$, and equality would make $R$ and $R_0$ carry the same partition, leaving no proper containment between their banks---so, by laminarity, $X_0$ is contained in one bank of $R$ and disjoint from the other. In particular, whichever bank a face of $\mathcal R'$ chooses for a river, the chosen bank is nested with or disjoint from $X_0$, and every face of $\mathcal R'$ extends to a face of $\mathcal R$ by the choice $B_{R_0} = X_0$.

	Now let a face of $\mathcal R'$, with chosen banks $B_R$, extend instead by the choice $B_{R_0} = X_0^{\mathrm c}$. Of the three ways a bank $B_R$ can be compatible with $X_0^{\mathrm c}$, two are ruled out by the preceding paragraph---$B_R$ disjoint from $X_0^{\mathrm c}$ would put $B_R$ inside $X_0$, and $X_0^{\mathrm c} \subseteq B_R$ would put $B_R^{\mathrm c}$ there---leaving $B_R \subseteq X_0^{\mathrm c}$. Every river other than $R_0$ must therefore choose its bank disjoint from $X_0$, which determines the face completely; and this assignment is indeed a face, being the choice of Lemma~\ref{lem:faceexist} for $F_0 = X_0$, a bank innermost by minimality. Hence exactly one face of $\mathcal R'$ extends by $X_0^{\mathrm c}$, and the faces of $\mathcal R$ number $(|\mathcal R'| + 1) + 1 = |\mathcal R| + 1$ in total. An injection between finite sets of the same size is a bijection, which proves the lemma.
\end{proof}

Fix a face of $\mathcal R$, with chosen banks $B_R$, and set
\begin{equation}
	W_i \;=\; \sum\,\bigl\{\, w_R : C_i \in B_R \,\bigr\},
\end{equation}
the total width of the rivers separating $C_i$ from the face. The set $F = \{\, C_i : W_i = 0 \,\}$ of the circles lying \emph{on} the face may consist of several cyclic runs, or be empty; no river separates two of its circles, since a chosen bank contains neither. For a circle $C$ with center $p \in \operatorname{int} P$ and radius $\rho>0$, we say that $C$ \emph{can be added on the face} if
\begin{equation}\label{eq:multiadd}
	d_i(p) \;\ge\; \rho + W_i \quad \text{for all } i,
\end{equation}
and that it \emph{touches} $C_i$---directly if $W_i = 0$, across the rivers separating $C_i$ from the face if $W_i > 0$---when equality holds; in the latter case the segment $\overline{pv_i}$ has length $\rho + W_i + r_i$, matching the combined width of the rivers crossing it. Subdividing along the segments from $p$ to the touched centers then decomposes $P$ as in Lemma~\ref{lem:split}, with the count of \eqref{eq:gap} unchanged.

Definition~\ref{def:multiring} requires both banks of every river to contain at least two circles. This loses no generality. A river $R$ whose one bank would contain a single circle $C_i$ can be \emph{collapsed} into it, enlarging $r_i$ by $w_R$ and deleting $R$ while keeping all the markers, which decreases $\sigma_{ij}$ by exactly $w_R$ for every $j \ne i$ and hence preserves both conditions of the definition; moreover, every face of the collapsed family is also a face of the original one, completed by choosing for $B_R$ the bank of $C_i$---by laminarity some bank of $R$ is properly contained in some bank of any other river, and it must be the bank of $C_i$, no bank holding two circles fitting inside it; the bank of $C_i$ is thus contained in one bank, hence contained in or disjoint from each bank, of every other river---with the same widths $W_j$ for $j \ne i$, while $d_i$ and $W_i$ decrease by the same amount $w_R$, so the circles that can be added on the face, and the circles they touch, are unchanged. This normalization is needed when the subdivision below is iterated, as a subdivided part may retain only one circle of some bank.

\begin{theorem}\label{thm:multiriver}
	For any $N$-gon of touching circles with rivers, with $N>4$, and any face of $\mathcal R$, it is possible to add a circle on that face, touching at least three of the circles---directly or across the rivers---so that the largest reduced polygon has, at most, $N-1$ sides.
\end{theorem}

The proof retraces Section~\ref{sec:rivers}. The rivers are absorbed into the circles, and the medial axis of what remains of the free region provides the added circle. The new difficulty is that the total expansions $W_i$ differ from circle to circle, and a disk expanding past a slower, already halted one could plough through it and emit \emph{new} free boundary, which the monotonicity argument of Lemma~\ref{lem:shrink} cannot tolerate. The remedy is to absorb the rivers \emph{one at a time, innermost first}. Then all disks in motion expand at a common unit rate at every moment, and we will see that they touch nothing else. We emphasize that the induction implicit in this staging is structural induction along a forest of rivers, defined below---equivalently, on its height---and not on the number of rivers. Absorbing a river does not produce a smaller instance of Definition~\ref{def:multiring}, since the expanded circles of its bank overlap, so no statement about configurations with fewer rivers is ever available; what is propagated from stage to stage is instead the boundary invariant of Lemma~\ref{lem:multishrink} below. The four absorption lemmas of this section extend those of Section~\ref{sec:rivers} one for one---Lemmas~\ref{lem:separation} and~\ref{lem:mouth} generalize the two parts of Lemma~\ref{lem:integrity}, Lemma~\ref{lem:multishrink} generalizes Lemma~\ref{lem:shrink}, and Lemma~\ref{lem:multimain} generalizes Lemma~\ref{lem:main}---and each pair coincides when the river is single; the reader may lean on the correspondence throughout.

The order of absorption is governed by a forest. The chosen banks are pairwise nested or disjoint; ordered by inclusion of these sets, the rivers thus form a forest of rooted trees, one for each maximal chosen bank, which we call the \emph{river forest} of the face, with $R'$ a \emph{descendant} of $R$ when $B_{R'} \subseteq B_R$ (parallel rivers, with equal banks, may be arranged in a chain arbitrarily); in the language of the circle/river method, this is the tree of the underlying base, rooted at the region of the face. Now order the rivers as $R_1, \dots, R_K$ so that every river comes after its descendants, that is, $B_{R_k} \subsetneq B_{R_l}$ implies $k<l$---for instance, by peeling the leaves of the river tree round after round---and abbreviate $B_k = B_{R_k}$, $w_k = w_{R_k}$. The choice of order gives the property used throughout:
\begin{equation}\label{eq:nested}
	\text{for } l < k: \qquad B_l \subseteq B_k \quad\text{or}\quad B_l \cap B_k = \emptyset.
\end{equation}
Let $\tau_k = w_1 + \cdots + w_k$, and call the time interval $[\tau_{k-1}, \tau_k]$ \emph{stage $k$}. During stage $k$ exactly the circles of $B_k$ expand, at unit rate. Thus for $t = \tau_{k-1} + s$ with $s \in [0, w_k]$,
\begin{equation}
	\tilde r_i(t) \;=\; r_i + a_i(t),
	\qquad
	a_i(t) \;=\; \sum_{l<k,\ C_i \in B_l} w_l \;+\; s\,[\,C_i \in B_k\,],
\end{equation}
where $[\,\cdot\,]$ is the Iverson bracket, which equals $1$ if the condition holds and $0$ otherwise. At the final time $\tau_K$ every circle has absorbed exactly the rivers separating it from the face, so $\tilde r_i = r_i + W_i$, and by \eqref{eq:multiadd} a circle can be added on the face exactly when it can be added to this final expanded configuration in the sense of Definition~\ref{def:add}. Let $\FR(t)$ denote the free region of the expanded disks, as before. On each edge $e = \overline{v_cv_{c+1}}$, the \emph{residual mouth} at time $t$ is the set of points $q \in e$ with $\dist{qv_c} \ge \tilde r_c(t)$ and $\dist{qv_{c+1}} \ge \tilde r_{c+1}(t)$---a segment of length $w_e - a_c(t) - a_{c+1}(t)$ when this is positive, the part of the gap not yet closed.

The next two lemmas confine each stage---an expanding disk can meet neither a resting disk nor a residual mouth. Both proofs amount to bookkeeping with \eqref{eq:nested}.

\begin{lemma}[separation]\label{lem:separation}
	During stage $k$, at time $t = \tau_{k-1}+s$, for all $C_j \in B_k$ and $C_m \notin B_k$,
	\[
		\dist{v_jv_m} - \tilde r_j(t) - \tilde r_m(t) \;\ge\; w_k - s,
	\]
	with equality only if $C_j, C_m$ are consecutive and every river crossing their edge, other than $R_k$, is already absorbed. In particular, an expanding disk meets a resting disk only at the very end of a stage, in a tangency of a consecutive pair.
\end{lemma}

\begin{proof}
	We claim that the rivers absorbed before stage $k$ that have expanded $C_j$ or $C_m$, together with $R_k$, are pairwise distinct and all separate $C_j$ from $C_m$. Distinctness is clear, since a river expands the circles of only one of its banks, and $R_k$ is not yet absorbed. If $R_l$ ($l<k$) has $C_j \in B_l$, then $B_l \subseteq B_k$ by \eqref{eq:nested}, so $C_m \notin B_l$ and $R_l$ separates the pair. If $R_l$ has $C_m \in B_l$, then $B_l \subseteq B_k$ is impossible, so $B_l \cap B_k = \emptyset$ by \eqref{eq:nested}, so $C_j \notin B_l$ and again $R_l$ separates. Finally $R_k$ separates, since $C_j \in B_k \not\ni C_m$. Consequently
	\[
		\sigma_{jm} \;\ge\; \bigl(a_j(t) - s\bigr) + a_m(t) + w_k,
	\]
	and combining with $\dist{v_jv_m} \ge r_j + r_m + \sigma_{jm}$ gives the bound. For non-consecutive pairs the latter inequality is strict. For consecutive pairs it is an equality with $\sigma_{jm} = w_e$, the rivers separating the pair being exactly those crossing their edge $e$; equality in the bound then forces every unabsorbed river crossing $e$ other than $R_k$ to be absent.
\end{proof}

\begin{lemma}[mouth integrity]\label{lem:mouth}
	At no time does a disk $\tilde D_j(t)$ contain a point of the residual mouth of an edge not incident to $v_j$.
\end{lemma}

\begin{proof}
	Since disks never shrink and residual mouths never grow, a violation with $C_j$ resting entails a violation at the last prior instant at which $C_j$ was expanding (or at $t=0$, the start of stage $1$); so we may assume $C_j \in B_k$ during stage $k$, at time $t = \tau_{k-1}+s$. Let $q$ be a point of the residual mouth of the edge $e = \overline{v_cv_{c+1}}$ with $j \ne c, c+1$, and suppose $q \in \tilde D_j(t)$. We distinguish three cases by the position of the endpoints of $e$.

	\emph{Both $C_c, C_{c+1} \in B_k$.} Then the residual mouth of $e$ is empty, a contradiction. Indeed, the bank $B_R$ of a river crossing $e$ contains exactly one of $C_c, C_{c+1}$, so it neither contains $B_k$ nor is disjoint from it nor equals it, and, the chosen banks being pairwise nested or disjoint, $B_R \subsetneq B_k$; by the choice of order every river crossing $e$ is then absorbed, each having expanded $C_c$ or $C_{c+1}$, so that $a_c(t) + a_{c+1}(t) \ge w_e$.

	\emph{Exactly one endpoint, say $C_c$, lies in $B_k$} (so that $R_k$ crosses $e$). Using $\dist{qv_c} \ge \tilde r_c(t)$,
	\[
		\dist{v_jv_{c+1}} \le \dist{v_jq} + \dist{qv_{c+1}} \le \tilde r_j(t) + \dist{v_cv_{c+1}} - \tilde r_c(t)
		= r_j + r_{c+1} + a_j(t) + w_e - a_c(t).
	\]
	We claim that $\sigma_{j,c+1} \ge a_j(t) + w_e - a_c(t)$; note that the two current-stage contributions $s$ cancel here, as both $C_j$ and $C_c$ lie in $B_k$. Granting the claim, the display contradicts Definition~\ref{def:multiring}\ref{it:multidisjoint} if $C_j, C_{c+1}$ are non-consecutive. If they are consecutive ($j = c+2$), the display and the claim force equality throughout; in particular $q$ lies on the segment $\overline{v_{c+2}v_{c+1}}$, and so on two distinct edges of $P$---its own edge $e$ and this one---which meet only at their shared vertex; thus $q = v_{c+1}$, impossible for a point of the residual mouth, as $\tilde r_{c+1} > 0$. To verify the claim, say that a river \emph{has expanded} $C_i$ if it is already absorbed with $C_i$ in its bank $B$; we check that each river $S \in \mathcal R$ contributes to the right-hand side of the claim at most what it contributes to $\sigma_{j,c+1}$, i.e., that
	\[
		[S \text{ has expanded } C_j] + [S \text{ crosses } e] - [S \text{ has expanded } C_c]
		\le [S \text{ separates } C_j, C_{c+1}].
	\]
	Only the cases with positive left-hand side need checking; all use \eqref{eq:nested} and $C_{c+1} \notin B_k$. If $S$ has expanded $C_j$, then $B_S \subseteq B_k$, so $C_{c+1} \notin B_S$ and $S$ separates $C_j$ from $C_{c+1}$; moreover, if such an $S$ crosses $e$, then the endpoint it contains must be $C_c$, and the left-hand side is again $1$. If $S$ crosses $e$ and has not expanded $C_j$, then either $C_{c+1} \in B_S$, in which case $B_S$ is neither nested with $B_k$ (it misses $C_c \in B_k$ and contains $C_{c+1} \notin B_k$)---so it is disjoint from $B_k$, misses $C_j$, and $S$ separates; or $C_c \in B_S$, in which case $B_S \subsetneq B_k$ would mean that $S$ has expanded $C_c$ (left-hand side $\le 0$), so $B_S \supseteq B_k$, the chosen banks being pairwise nested or disjoint, so $C_j \in B_S \not\ni C_{c+1}$ and $S$ separates.

	\emph{Neither endpoint lies in $B_k$.} Then $R_k$ does not cross $e$ and separates $C_j$ from both endpoints; so does every river absorbed with $C_j$ in its bank, whose bank is contained in $B_k$ by \eqref{eq:nested} and therefore avoids $C_c, C_{c+1}$. Since, moreover, every river crossing $e$ separates $C_j$ from exactly one endpoint,
	\[
		\sigma_{jc} + \sigma_{j,c+1} \;\ge\; 2\bigl(a_j(t) - s\bigr) + 2w_k + w_e \;\ge\; 2a_j(t) + w_e.
	\]
	On the other hand, adding the triangle estimates through the two endpoints and using $\dist{qv_c} + \dist{qv_{c+1}} = \dist{v_cv_{c+1}} = r_c + r_{c+1} + w_e$,
	\[
		2\dist{v_jq} \;\ge\; \dist{v_jv_c} + \dist{v_jv_{c+1}} - r_c - r_{c+1} - w_e
		\;\ge\; 2r_j + \sigma_{jc} + \sigma_{j,c+1} - w_e
		\;\ge\; 2\tilde r_j(t),
	\]
	where the middle inequality is strict because $C_j$ cannot be consecutive to both endpoints. Hence $q \notin \tilde D_j(t)$.
\end{proof}

With the stages thus insulated, the shrinking lemma extends.

\begin{lemma}[shrinking, staged]\label{lem:multishrink}
	For every $t \in [0, \tau_K]$, each connected component of $\FR(t)$ is bounded by circular arcs of distinct circles---at most one arc per circle---occurring along the boundary in their cyclic index order, interleaved in the correct positions with residual mouth segments, at most one per edge.
\end{lemma}

\begin{proof}
	At $t = 0$ the proofs of Lemmas~\ref{lem:cover} and~\ref{lem:domain} apply verbatim except on the edges with $w_e > 0$, where the mouths appear. Thus $\FR(0)$ is a single component bounded by the $N$ arcs in cyclic order and the initial mouths. We show that this structure can only degenerate, never reorganize.

	First, death is permanent, across all stages. Fix $C_j$ and a unit vector $u$, and follow the point $p(t)$ with $\overrightharpoon{p}(t) = \overrightharpoon{v_j} + \tilde r_j(t)\,u$, which moves, at unit speed, exactly while $C_j$ expands. Suppose $p(t_0) \in \tilde D_m(t_0)$ for some $m \ne j$, and take $t_0$ minimal. At $t_0$, at least one of $C_j$ and $C_m$ is expanding, for otherwise both $p$ and $\tilde D_m$ would be stationary just before $t_0$, contradicting the minimality of $t_0$. In fact both are expanding, since the point $p(t_0)$ lies in $\tilde D_j(t_0) \cap \tilde D_m(t_0)$, and by Lemma~\ref{lem:separation} an expanding disk meets no resting disk. So $C_j, C_m \in B_k$ for the current stage $k$. At any later time at which $C_j$ expands, during a stage $l \ge k$ with $C_j \in B_l$, property \eqref{eq:nested} gives $B_k \subseteq B_l$, so $C_m$ expands as well; therefore $\dist{p(t)\,v_m} \le \tilde r_m(t)$ persists exactly as in the proof of Lemma~\ref{lem:shrink}---while both circles expand, $p$ moves with unit speed and $\tilde D_m$ grows with unit speed, and while $C_j$ rests, $p$ is stationary and disks never shrink. If instead $p(t_0)$ exits $P$, then at the first later return of $p$ to $\partial P$ the return point lies on an edge and on no residual mouth (it lies on $\tilde C_j \subset \tilde D_j$, and Lemma~\ref{lem:mouth} applies, the incident edges being met only along themselves), and is therefore within $\tilde r_c$ of an endpoint $v_c$ of its edge, and we are back in the previous case, with $\tilde D_c$ as the covering disk. Hence every boundary arc---viewed as the set of directions $u$ it spans, its radius never decreasing---can only shrink, split, or vanish, and no circle ever contributes new boundary material; mouths likewise only recede, their endpoints advancing with the expanding end disks, unreachable by all other disks.

	Second, the combinatorial structure changes only at finitely many critical times, namely when an arc or mouth degenerates to a point, or when two expanded circles become tangent at a point of the closure of the free region. By Lemma~\ref{lem:separation} such a tangency involves two co-expanding circles, or a consecutive pair at the end of a stage, where their residual mouth closes into a cusp. A visible tangency point of two co-expanding circles lies in the interior of $P$; on $\partial P$ it would be a covered edge point, invisible, or a point of a residual mouth, excluded by Lemma~\ref{lem:mouth}. The pinching argument of Lemma~\ref{lem:shrink} then applies verbatim---the component splits, and each part inherits a contiguous run of the boundary sequence. In all cases the invariant survives, and the lemma follows by induction over the critical times; the transitions between stages are not events, the radii being continuous.
\end{proof}

\begin{lemma}\label{lem:multimain}
	Let $M_1, \dots, M_d$ be the maximal chosen banks (pairwise disjoint), call the edges joining a circle of some $M_a$ to a circle outside it the \emph{junction edges}, and call the two circles of $M_a$ adjacent to its two junction edges its \emph{extremes}. At $t = \tau_K$, every junction edge carries a tangency cusp of its two endpoint disks, and some component $\FR_\ast$ of $\FR(\tau_K)$ carries on its boundary all the arcs of the circles of $F$, all the junction cusps, and, for each $M_a$, arcs of both of its extremes.
\end{lemma}

\begin{proof}
	Throughout the proof we use the following \emph{strict separation}. For $C_i \in M := M_a$ and $C_m \notin M$, every river whose chosen bank contains $C_i$ or $C_m$ separates the pair---in the first case the chosen bank is nested into $M$, which avoids $C_m$; in the second it can neither contain nor equal the maximal bank $M$, so it is disjoint from it and avoids $C_i$---so $\sigma_{im} \ge W_i + W_m$ and, at every time,
	\[
		\dist{v_iv_m} - \tilde r_i - \tilde r_m \;\ge\; \sigma_{im} - W_i - W_m \;\ge\; 0,
	\]
	strictly for non-consecutive pairs. In particular, the disks of $M$ never meet the disks of the other maximal banks, nor those of the circles of $F$, except that a consecutive pair across a junction edge may become tangent.

	\emph{Cusps.} Let $e = \overline{v_cv_{c+1}}$ be a junction edge, say $C_{c+1} \in M$ and $C_c \notin M$. A chosen bank containing $C_{c+1}$ is a subinterval of $M$ containing its extreme $C_{c+1}$, so $e$ is one of its two boundary edges and its river crosses $e$; likewise, a chosen bank containing $C_c$ is a subinterval, with extreme $C_c$, of the maximal bank on the other side (if $C_c \in F$ there is none). Hence every river that ever expands $C_c$ or $C_{c+1}$ crosses $e$; consequently $W_c + W_{c+1} = w_e$, the residual mouth of $e$ has, at every time, length equal to the total width of its unabsorbed rivers, and it closes into a tangency of $\tilde D_c$ and $\tilde D_{c+1}$ exactly when the last of them is absorbed---after which neither endpoint circle ever expands again, so the tangency persists to $\tau_K$. Its cusp point stays on the boundary of the free region. It is a limit of points of the residual mouth, and a closed disk containing no point of a residual mouth contains no limit of such points (Lemma~\ref{lem:mouth}).

	\emph{Fronts.} The goal of this step is to show that each maximal bank presents a single connected \emph{front} to the face---the union of its expanded disks borders the free region along one chain of free arcs, running from the cusp on one of its junction edges to the cusp on the other---so that, in particular, both extremes of the bank contribute arcs. Fix $M = M_a$, let $x, x'$ be the cusps on its two junction edges, and let
	\[
		Y \;=\; \bigcup_{C_i \in M} \bigl(\tilde D_i(\tau_K) \cap \overline P\bigr).
	\]
	First, $Y$ is connected. For consecutive circles of $M$, the point of their edge at distance $\tilde r_c$ from $v_c$ lies in both disks, since every river crossing an edge within $M$ is absorbed (its chosen bank, containing one endpoint, is a proper subinterval of $M$); so the sets $\tilde D_i \cap \overline P$, $C_i \in M$, chain up through their edges. And every point $y \in \tilde D_i \cap \overline P$ joins this chain within $Y$. If the segment $\overline{yv_i} \subset \tilde D_i$ leaves $P$, then, walking from $y$ toward $v_i$, its first crossing $z$ of $\partial P$ lies on an edge, in $\tilde D_i$; either that edge is incident to $v_i$ and the piece of it from $z$ to $v_i$ lies in $\tilde D_i \cap \overline P$, or $z$ lies on no residual mouth (Lemma~\ref{lem:mouth}) and is therefore covered by an endpoint disk $\tilde D_c$ of its edge, with $C_c \in M$ by strict separation, and the piece of the edge from $z$ to $v_c$ lies in $\tilde D_c \cap \overline P$.

	Second, $Y \cap \partial P$ is the closed stretch of $\partial P$ from $x$ to $x'$ behind $M$. It contains the covered parts of the two junction edges up to their cusps, the edges within $M$ (whose mouths are closed, as above), and the centers of the circles of $M$; conversely, no disk of $M$ contains the center of another circle (by strict separation, $\dist{v_iv_m} \ge \tilde r_i + \tilde r_m > \tilde r_i$), nor a point of an edge joining two circles outside $M$---such a point lies on a residual mouth, excluded by Lemma~\ref{lem:mouth}, or is covered by an endpoint disk, which the disks of $M$ avoid, junction pairs touching only at their cusps, on their own edges.

	Now let $V$ be the component of $\overline P \setminus Y$ containing the complementary stretch $\partial P \setminus (Y \cap \partial P)$, which is connected. The set $\overline P \setminus V$ is closed and connected, being the union of $Y$ with the remaining components of $\overline P \setminus Y$, the closure of each of which meets $Y$. Since $\overline P$ is homeomorphic to a closed disk, it is \emph{unicoherent}: whenever it is written as the union of two closed connected sets, their intersection is connected \cite[\S57]{Kuratowski1968}. Applied to $\overline V$ and $\overline P \setminus V$, whose intersection is the frontier of $V$, this shows that frontier to be connected. This frontier contains $x$ and $x'$, lies in $Y$, and meets $\partial P$ only at these two points; elsewhere it consists of points of the circles $\tilde C_i$, $C_i \in M$, interior to $P$, interior to no disk of $M$ (they lie on the boundary of $Y$) and to no other expanded disk (strict separation)---that is, of free boundary arcs of the circles of $M$, which near $x$ and $x'$ run along the disks of the two extremes.

	\emph{Assembly.} The circles of $F$ never expand, and no expanded disk ever meets their disks or arcs (strict separation, applied to the maximal bank of the expanded circle). The free points at sufficiently small positive distance from the arcs of the circles of $F$ and from the fronts just constructed form a connected set. Within each cyclic run of $F$, the collar passes through the interior wedges at the tangency cusps of the run, as in Lemma~\ref{lem:main}; and at each junction cusp, the interior wedge between the two tangent disks joins the collars or fronts on its two sides. Hence all of these free points lie in a single component $\FR_\ast$ of $\FR(\tau_K)$, whose boundary carries the arcs of all circles of $F$, every junction cusp, and arcs of both extremes of every maximal bank.
\end{proof}

\begin{proof}[Proof of Theorem~\ref{thm:multiriver}]
	Let $\FR_\ast$ be the component of Lemma~\ref{lem:multimain} and $\MA = \MA(\FR_\ast)$. By Lemma~\ref{lem:multishrink}, $\partial\FR_\ast$ consists of arcs of distinct circles in cyclic index order, meeting at corners which are tangency cusps or transversal crossings. No residual mouth segment survives at $\tau_K$. An edge joining two circles of $F$ is crossed by no river; a junction edge has closed into its cusp (Lemma~\ref{lem:multimain}); and every other edge joins two circles of the same maximal bank, so all the rivers crossing it are absorbed---their chosen banks being proper subintervals---each absorption expanding an endpoint circle of the edge, so that $a_c + a_{c+1} \ge w_e$ in the end. Crossings are indeed transversal, no two of the final circles being nested:
	\[
		|\tilde r_i - \tilde r_j| \;\le\; |r_i - r_j| + |W_i - W_j| \;<\; r_i + r_j + \sigma_{ij} \;\le\; \dist{v_iv_j},
	\]
	using $|W_i - W_j| \le \sigma_{ij}$, which holds because the rivers counted in one of $W_i, W_j$ but not the other separate $C_i$ from $C_j$. All corners are therefore sharp, $\FR_\ast$ is a simply connected domain of the class of Theorem~\ref{thm:ccm}, and the analysis of Lemma~\ref{lem:tree} applies with the modifications noted in the proof of Theorem~\ref{thm:river}. Thus $\MA$ is a finite tree whose leaves are the corners of $\partial\FR_\ast$, and every non-leaf vertex, of degree $k$, is the center of a maximal disk touching exactly $k$ of the final circles, with $d_i(p) = \rho + W_i$ at the touched circles and $d_i(p) \ge \rho + W_i$ at all others---by \eqref{eq:multiadd}, a circle that can be added on the face, touching the corresponding circles directly or across the rivers.

	The two cases of the proof of Theorem~\ref{thm:river} now go through with $F$ in place of the bank $A$. Here the circles of $F$ and the two extremes of every maximal bank are never hidden (Lemma~\ref{lem:multimain}), and these number at least four---if there is a single maximal bank, then $F$ contains the other bank of its river, which has at least two circles. When at least five circles contribute arcs, the first case applies verbatim, its accounting of gaps against leaves making no reference to the banks. When exactly four circles contribute, the hidden circles need no longer form a single run. They do lie in the interiors of the maximal banks, everything else---the circles of $F$ and the extremes---being never hidden, and four contributors leave room for the extremes of at most two maximal banks; the hidden circles therefore fill at most two runs. So we take instead the corner between the two contributing circles enclosing a \emph{largest} hidden run, of $h \ge 1$ circles, and the non-leaf vertex $I$ reached from it. The gap of $I$ through that run is $h + 1 \le N-3$, and every other branch of $I$ contains at most two of the remaining leaves and at most the remaining $N - 4 - h$ hidden circles, so every other gap is at most $2 + (N-4-h) \le N-3$.

	It remains to check that the subdivision is valid and that the rivers remain realizable. For the subdivision, the segment from $p$ to a touched center $v_i$ passes through the contact point $q_i$, and $\overline{pq_i} \subset \mathcal{D}_p \subseteq \FR_\ast \subseteq P$; the remaining piece $\overline{q_iv_i}$ is the radius of $\tilde D_i$ in the direction of $q_i$, and lies in $P$. Since $q_i$ is a free boundary point at the final time, the point $p(t)$ of its direction never left $P$---a direction that exits $P$ is dead forever by the proof of Lemma~\ref{lem:multishrink}---so the radius from $v_i + r_iu$ outward is swept by $p(t)$ within $P$, while the initial piece inside the original disk $D_i$ is a radial segment of the sector $D_i \cap P$ of Lemma~\ref{lem:domain}\ref{it:sector}, whose proof applies to the initial configuration verbatim. The segments therefore subdivide $P$ as in Lemma~\ref{lem:split}, into polygons with $g_j + 2 \le N-1$ sides. Each reduced polygon, endowed with the markers on its own edges and the rivers of $\mathcal R$ that separate two of its circles---the added circle $C$ counting as lying on the face, so that a river separates $C$ from $C_i$ exactly when $C_i \in B_R$, and a river left with a single circle of the part on one side being collapsed into it as described after \eqref{eq:multiadd}---is again a polygon of touching circles with rivers (or without, when no river survives); its new edges have length $\rho + r_i + W_i$, matching condition~\ref{it:multitangent}, since the rivers crossing the new edge $\overline{pv_i}$ are exactly those separating $C_i$ from the face; condition~\ref{it:multidisjoint} is inherited, and so is laminarity---restriction to the part preserves every containment, and where it turns a proper containment into an equality, everything telling the two rivers apart having fallen to the other parts, a fresh marker inserted between them, in the order of their original nesting, restores it. Finally, the rivers can be rerouted stage by stage. The course of $R_k$ may be taken to run along the strip swept during stage $k$ by the free-side boundary of $\bigcup_i \tilde D_i(t)$, a strip of width $w_k$ entering and leaving $P$ through the residual mouths of the two edges crossed by $R_k$, which recede at unit rate during exactly this stage. These strips are pairwise disjoint, the union of the disks being nondecreasing and the time windows disjoint; each separates precisely the circles of $B_k$ from the others; they avoid the disks of $F$ (Lemma~\ref{lem:separation}) and keep distance at least $\rho$ from the added circle, since $d_i(p) \ge \rho + W_i \ge \rho + a_i(t)$ for all $i$ and $t$.
\end{proof}

\begin{corollary}\label{cor:multibound}
	Every $N$-gon of touching circles with rivers can be subdivided into triangles and quadrilaterals by adding at most $N-4$ circles.
\end{corollary}

\begin{proof}
	The induction of Corollary~\ref{cor:bound} applies verbatim, using Theorem~\ref{thm:multiriver} (with any face of the inherited river family) in place of Theorem~\ref{thm:main} for the parts crossed by rivers.
\end{proof}

\begin{remark}
	Theorem~\ref{thm:multiriver} generalizes the freedom of choice of Theorem~\ref{thm:river}. The designer may place the added flap in any region of the design visible within the polygon---in a leaf group of flaps, between rivers, or even in a region carrying no flap at all, such as the central region enclosed by three rivers whose far banks are pairwise disjoint---one appears in Figure~\ref{fig:multiriver}---in which case every touch of the added circle is across a river.
\end{remark}


\section{Concluding remarks}
\label{sec:conclusion}

We have shown that in any axial polygon of a circle-packed design with more than four sides, a circle can always be added so that every reduced polygon is strictly smaller, and consequently that at most $N-4$ added circles reduce an $N$-gon entirely into triangles and quadrilaterals---for which classical geometric molecule constructions exist. The proof identifies the candidate circles with the internal vertices of the medial axis of the free region, a tree whose combinatorics alone guarantee the existence of a good vertex; the same skeleton, computed as an Apollonius diagram, serves as a practical algorithm for finding the added circle.

Several questions remain open.
\begin{enumerate}
	\item \emph{Minimum number of additions.} The bound of $N-4$ added circles is sharp in the worst case (Corollary~\ref{cor:tight}), but an individual configuration may admit a much shorter subdivision---a single circle touching many others can reduce an $N$-gon to triangles at one stroke. Our algorithm only guarantees a decrease of the order at each step and does not compute the minimum; characterizing it, or computing it efficiently, is open.
	\item \emph{Other objectives.} From the design point of view, each added circle is an extra flap, so one may prefer to minimize quantities other than their number---the total radius of the added circles (the total length of extra material), the largest added radius, or a subdivision using only radii from a prescribed range. Nothing is known about these constrained versions, including whether they are computationally tractable.
	\item \emph{Quality of the molecules.} The interplay between the choice of the added circle and the shape of the resulting molecules---for instance, the aspect ratios of the triangles and quadrilaterals produced---is a natural direction for further study, as thin polygons yield crease patterns that are difficult to fold in practice.
\end{enumerate}
Beyond these, the staged-shrinking technique of Lemmas~\ref{lem:shrink} and \ref{lem:multishrink} may be of independent interest for further extensions, for instance to metrics other than the Euclidean one or to more general weighted Voronoi settings.


\section*{Acknowledgement}

The author extends gratitude to Robert J. Lang for inspiring this work and engaging in fruitful conversations throughout the journey.


\bibliographystyle{osmebibstyle}
\bibliography{osmerefs}

\theaffiliations

\end{document}